\documentclass[12pt]{article}
\usepackage{amsthm}
\usepackage{amsmath}
\usepackage{amsfonts}
\usepackage{graphicx}
\usepackage{latexsym}
\usepackage{amssymb}

\DeclareMathAlphabet{\mathbfsl}{OT1}{ppl}{b}{it} 

\newcommand{\deff}{\mbox{$\stackrel{\rm def}{=}$}}
\newcommand{\Span}[1]{{\left\langle {#1} \right\rangle}}

\makeatletter
 \DeclareRobustCommand{\nsbinom}{\genfrac[]\z@{}}
 \makeatother
 \newcommand{\sbinom}[2]{\nsbinom{{#1}}{{#2}}}
 \newcommand{\sbinomq}[2]{\nsbinom{{#1}}{{#2}}_{q}}
  \newcommand{\sbinomtwo}[2]{\nsbinom{{#1}}{{#2}}_{2}}

\newcommand{\field}[1]{\mathbb{#1}}
\newcommand{\A}{\field{A}}
\newcommand{\dB}{\field{B}}

\newcommand{\F}{\field{F}}

\newcommand{\dS}{\field{S}}
\newcommand{\T}{\field{T}}

\newcommand{\cN}{{\cal N}}

\newcommand{\cG}{{\cal G}}

\newcommand{\C}{{\mathbb C}}

\newcommand{\linadd}{\kern1pt\mbox{\small$\boxplus$}\kern1pt}

\newtheorem{defn}{Definition}
\newtheorem{theorem}{Theorem}
\newtheorem{lemma}{Lemma}

\newtheorem{cor}{Corollary}

\newtheorem{example}{Example}

\begin{document}

\bibliographystyle{plain}

\title{
\begin{center}
A New Approach to Examine $q$-Steiner Systems
\end{center}
}
\author{
{\sc Tuvi Etzion}\thanks{Department of Computer Science, Technion,
Haifa 32000, Israel, e-mail: {\tt etzion@cs.technion.ac.il}.}}

\maketitle

\begin{abstract}
One of the most intriguing problems, in $q$-analogs of designs and codes, is the existence question of
an infinite family of $q$-analog of Steiner systems (spreads not included) in general,
and the existence question for the $q$-analog of the Fano plane in particular.

We exhibit a completely new method to attack this problem. In the process we define
a new family of designs whose existence is implied by the
existence of $q$-Steiner systems, but could exist even if the
related $q$-Steiner systems do not exist.

The method is based on a possible system obtained by puncturing
all the subspaces of the $q$-Steiner system several times.
We define the punctured system as a new type of design and
enumerate the number of subspaces of various types that it might have.
It will be evident
that its existence does not imply the existence of the related $q$-Steiner system.
On the other hand, this type of design demonstrates how close can we get to
the related $q$-Steiner system.

Necessary conditions for the existence
of such designs are presented. These necessary conditions will be also necessary conditions
for the existence of the related $q$-Steiner system.
Trivial and nontrivial direct constructions
and a nontrivial recursive construction for such designs, are given.
Some of the designs have a symmetric structure, which is uniform in
the dimensions of the existing subspaces in the system. Most constructions are based
on this uniform structure of the design or its punctured designs.
Finally, the structure of the $q$-Fano plane for any given $q$, was
considered based on this new approach.
\end{abstract}

\vspace{0.5cm}

\noindent {\bf Keywords:} puncturing, $q$-analog, spreads, $q$-Fano plane, $q$-Steiner systems.

\vspace{0.5cm}

\noindent
{\bf Mathematics Subject Classification}: 94B25, 05B40, 51E10  .

\footnotetext[1] { This research was supported in part by the Israeli
Science Foundation (ISF), Jerusalem, Israel, under
Grant 10/12.}

\newpage
\section{Introduction}

Let $\F_q$ be the finite field with $q$ elements
and let $\F_q^n$ be the set of all vectors
of length $n$ over~$\F_q$. $\F_q^n$ is a vector
space with dimension $n$ over $\F_q$. For a given integer $k$,
$0 \leq k \leq n$, let $\cG_q(n,k)$ denote the set of all
$k$-dimensional subspaces ($k$-\emph{subspaces} in short) of $\F_q^n$. $\cG_q(n,k)$ is often
referred to as a Grassmannian. It is well known that
$$ \begin{small}
| \cG_q (n,k) | = \sbinomq{n}{k}
\deff \frac{(q^n-1)(q^{n-1}-1) \cdots
(q^{n-k+1}-1)}{(q^k-1)(q^{k-1}-1) \cdots (q-1)}
\end{small}
$$
where $\sbinomq{n}{k}$ is the $q$-\emph{binomial coefficient} (known also
as the \emph{Gaussian coefficient}~\cite[pp. 325-332]{vLWi92}).

A \emph{Grassmannian code} (known better as a \emph{constant dimension code})
$\C$ is a subset of
$\cG_q(n,k)$. In recent years there has been an increasing
interest in Grassmannian codes as a result of their
application to error-correction in random network coding which was
demonstrated in the seminal work by Koetter and Kschischang~\cite{KoKs08}.
This work has motivated lot of research on coding for
Grassmannian codes (see for example~\cite{EtSi13,EtVa11}
and references therein). But, the
interest in these codes has been also before this application,
since Grassmannian codes are the $q$-analogs of the well
studied constant weight codes~\cite{BSSS}. The Grassmann scheme is the $q$-analog of
the Johnson scheme, where $q$-analogs replace concepts of subsets by concepts of subspaces when
problems on sets are transferred to problems on subspaces
over the finite field $\F_q$.
For example, the size of a set is replaced by the dimension of a subspace,
the binomial coefficients are replaced by the Gaussian coefficients, etc.
One example of such $q$-analog problem in coding theory is the nonexistence of nontrivial perfect codes
in the Grassmann scheme which was proved in~\cite{Chi87,MaZh95}. This problem is the $q$-analog for the nonexistence
problem of perfect codes in the Johnson scheme, which is a well-known
open problem~\cite{Del73,Etz96,EtSc04}. Also, the $q$-analogs of other various
combinatorial objects are well known~\cite[pp. 325-332]{vLWi92}.
The work of Koetter and Kschischang~\cite{KoKs08} has motivated also
increasing interest and lot of research work on these related $q$-analog of designs
(see for example~\cite{BEOVW,Etz14,EtVa11a} and references therein).
The most intriguing question is the existence of $q$-analog for Steiner system
which is the topic of the research in this paper.

A \emph{Steiner system} $S(t,k,n)$ is a set $S$ of $k$-subsets (called \emph{blocks})
from an $n$-set~$\cN$ such that each $t$-subset of $\cN$ is contained
in exactly one block of $S$. Steiner systems were subject to an
extensive research in combinatorial designs~\cite{CoDi07}.
A Steiner system is
also equivalent to an optimal constant weight code in
the Hamming scheme. It is well-known that if a Steiner system
$S(t,k,n)$ exists, then for all $0 \leq i \leq t-1$,
$\frac{\binom{n-i}{t-i}}{\binom{k-i}{t-i}}$ must be integers.
It was proved only recently that these necessary conditions for the
existence of a Steiner system $S(t,k,n)$ are also sufficient for each $t$ and $k$
such that $0 < t <k$, except for
a finite number of values of~$n$~\cite{Kee14}.

Cameron~\cite{Cam74,Cam74a} and Delsarte~\cite{Del76} have
extended the notions of block design and Steiner systems
to vector spaces.
A \emph{$q$-Steiner system} $\dS_q(t,k,n)$ is a set $\dS$ of
$k$-subspaces of $\F_q^n$ (called \emph{blocks}) such that each $t$-subspace of $\F_q^n$
is contained in exactly one block of $\dS$.
A $q$-Steiner system $\dS_q(t,k,n)$ is an optimal constant dimension
code~\cite{EtVa11,EtVa11a}. Similarly, to Steiner systems, simple necessary divisibility conditions
for the existence of a given $q$-Steiner system were developed~\cite{ScEt,Suz90}.

\begin{theorem}
\label{thm:derived}
If a $q$-Steiner system $\dS_q (t,k,n)$ exists, then
for each $i$, $1 \leq i \leq t-1$, a $q$-Steiner system
$\dS_q(t-i,k-i,n-i)$ exists.
\end{theorem}

\begin{cor}
\label{cor:ness}
If a $q$-Steiner system $\dS_q(t,k,n)$ exists, then for all $0 \leq i \leq t-1$,
$$
\frac{\sbinomq{n-i}{t-i}}{\sbinomq{k-i}{t-i}}
$$
must be integers.
\end{cor}

While a lot of information is known about the existence of Steiner systems~\cite{CoDi07,Kee14}, our knowledge
about the existence of $q$-Steiner systems is quite limited.
Until recently, the only known $q$-Steiner systems $\dS_q(t,k,n)$ were either trivial or
for $t=1$, where such systems exist if and only if $k$ divides $n$.
These systems are known as \emph{spreads} in finite geometries and they will be used in Section~\ref{sec:recursion}
and will be considered also in Sections~\ref{sec:structure} and~\ref{sec:2punctured}.
Thomas~\cite{Tho96} showed that certain kind of $q$-Steiner
systems $\dS_2(2,3,7)$ cannot exist. Metsch~\cite{Met99} conjectured
that nontrivial $q$-Steiner systems with $t \geq 2$ do not exist.
The concept of $q$-Steiner systems appeared also in connection of diameter perfect
codes in the Grassmann scheme. It was proved in~\cite{AAK01}
that the only diameter perfect codes in the Grassmann scheme
are the $q$-Steiner systems.
Recently, the first $q$-Steiner system $\dS_q(t,k,n)$
with $t \geq 2$ was found.
This is a $q$-Steiner system $\dS_2(2,3,13)$ which have a large
automorphism group~\cite{BEOVW}. Using $q$-analog of derived and residual designs
it was proved that sometimes the necessary conditions for the existence
of a $q$-Steiner system $\dS_q(t,k,n)$ are not sufficient~\cite{KiLa15}.
The first set of parameters ($t$, $k$, and $n$) for which the existence question of $q$-Steiner systems
is not settled is the parameters for the $q$-analog
of the Fano plane, i.e. the $q$-Steiner systems $\dS_q(2,3,7)$, which
will be called also in this paper the $q$-\emph{Fano plane}. There was a lot of effort to
find whether the $q$-Fano plane, especially for $q=2$, exists or does not exist, e.g.~\cite{BKN15,EtVa11a,HeSi16,Tho96}
All these attempts didn't provide any answer to the existence question. It was proved recently
in~\cite{BKN15} that if such system exists for $q=2$, then its automorphism group has a small order.

In this paper we present a completely new approach to examine the existence of $q$-Steiner systems.
This approach is based on the structure obtained by puncturing some coordinates from all the subspaces of the possible
$q$-Steiner system. This is equivalent to say that the projection of the other coordinates is considered
for all the subspaces of the system. This idea was suggested first in~\cite{Etz13} and this paper
completes and proves all the ideas mentioned in~\cite{Etz13}. The approach will involve
sizes of punctured subspaces and the numbers involved are sometimes the same as
those in~\cite{KiPa15}, where intersection numbers of combinatorial designs were considered.
But, the object considered in these two papers are different and the results are different.
We consider (and define) this structure, obtained by puncturing all the subspaces
of the system, as a new type of design, which exists if
the related $q$-Steiner system exists, but could exist even if the
related $q$-Steiner system does not exist.
If this design does not exist,
then the related $q$-Steiner system does not exist.
To highlight, our main contributions in this paper are:
\begin{enumerate}
\item A definition of a new method to examine the existence of a $q$-Steiner system.

\item A definition for a new type of designs which are close to $q$-Steiner systems
and their construction might be a first step to find the related $q$-Steiner system.
Some constructions for these designs are given.

\item An analysis of the $q$-Fano plane for any $q>1$.

\item Improving our understanding of the structure of the $q$-Fano plane for $q \geq 2$,
and in particular for $q=2$,
which we hope will help to find such a structure or prove its nonexistence.
\end{enumerate}

The rest of this paper is organized as follows. In Section~\ref{sec:punctured},
the definition of punctured $q$-Steiner systems and other related definitions,
are presented. We prove some properties of punctured
systems, define the new type of design, and examine some of its properties.
An inverse operation for puncturing and related operations are also presented and the
number of subspaces which are generated by the inverse operation are computed.
In Section~\ref{sec:system}, a system of equations, which form the necessary
conditions for the existence of the related punctured $q$-Steiner system, is presented. These systems of equations are
obtained by precise enumeration of covering $t$-subspaces by $k$-subspaces in the the
$q$-Steiner system as reflected by the punctured system. In Section~\ref{sec:examples},
a sequence of examples for punctured $q$-Steiner systems $\dS_q(k-1,k,n)$,
i.e. a sequence of examples for the new defined design, is presented.
In Section~\ref{sec:recursion}, a recursive construction for punctured $q$-Steiner systems
$\dS_q(2,3,n)$ is presented. One of the important ingredients, for this construction, is a large set of spreads.
In Section~\ref{sec:structure},
we prove our main results concerning the structure of the $q$-Fano plane,
for any power of a prime $q$.
In Section~\ref{sec:2punctured}, we consider one possible
structure of the twice punctured $q$-Fano plane. The given construction is only one
of a few possibilities to construct this design. In Section~\ref{sec:conclude}, we conclude with suggestions
for future research on the directions to advance the knowledge on this problem and maybe how to settle
it for good.

\section{Punctured $q$-Steiner Systems}
\label{sec:punctured}

Given an $n \times m$ matrix $A$, the \emph{punctured} matrix $A'$ is an $n \times (m-1)$
matrix obtained from $A$ by deleting one of the columns from $A$. Codes and punctured
codes in the Hamming space are well established in coding theory~\cite[pp. 27-32]{McSl77}.
A $q$-analog of punctured codes for subspace codes (codes whose codewords are
subspaces such as the Grassmannian codes),
was defined in~\cite{EtSi09}, but this is not the puncturing considered in this paper.

A subspace $X \in \cG_q(n,k)$, i.e. a $k$-subspace of $\F_q^n$, consists of
$q^k$ vectors of length $n$ with elements taken from $\F_q$.
The punctured subspace $X'$ by the $i$th coordinate is defined
as the subspace obtained from $X$ by deleting coordinate $i$ in all the vectors of $X$.
The result of this puncturing is a new subspace of
$\F_q^{n-1}$. If $X$ does not contain the unity vector with an \emph{one} in the $i$th coordinate, $e_i$,
then $X'$ is a subspace in $\cG_q(n-1,k)$. If $X$ contains the unity vector with a \emph{one} in the $i$th coordinate,
then $X'$ is a subspace in $\cG_q(n-1,k-1)$.
Assume that we are given a set $\dS$ of subspaces from $\cG_q(n,k)$. The \emph{punctured set}
$\dS'$ is defined as $\dS' = \{ X' ~:~ X \in \dS \}$, where all subspaces are
punctured in the same coordinate, and it can contain subspaces only from
$\cG_q(n-1,k)$ or from $\cG_q(n-1,k-1)$. The set $\dS'$ is regarded as multi-set and hence $|\dS'|=|\dS|$
(since two distinct $k$-subspaces of $\F_q^n$ can be punctured into the same $k$-subspaces of $\F_q^{n-1}$).
When the punctured coordinate is not mentioned it will be assumed that the last coordinate was punctured.
A subspace can be punctured several times. A $k$-subspace $X$, of $\F_q^n$, is
punctured $p$ times (i.e. $p$-\emph{punctured}) to a $p$-\emph{punctured subspace} $Y$
of $\F_q^{n-p}$. The subspace $Y$ can be an $s$-subspace for any $s$
such that $\max \{ 0, k-p \} \leq s \leq \min \{ k, n-p \}$.
Similarly we define a $p$-\emph{punctured set}. We summarize this brief
introduction on punctured subspaces with the main observation.

\begin{lemma}
\label{lem:1-punctured}
A $k$-subspace of $\F_q^n$, $k >0$, is punctured either into a $k$-subspace or into a $(k-1)$-subspace of $\F_q^{n-1}$.
A $k$-subspace of $\F_q^n$ is $p$-punctured into an $s$-subspace such that
$\max \{ 0, k-p \} \leq s \leq \min \{ k, n-p \}$.
\end{lemma}

If $\dS$ is a $q$-Steiner system $\dS_q(t,k,n)$, we would like to know if the $p$-punctured system $\dS'$
has some interesting properties, i.e. $\dS'$ has some uniqueness properties
related to the punctured $k$-subspaces and the punctured (contained)
$t$-subspaces. The motivation is to define a new set of designs which must exist
if the related $q$-Steiner systems $\dS_q(t,k,n)$ exist. But, these new designs can exist even if the
related $q$-Steiner systems do not exist. If the nonexistence of such designs can be proved,
then the related $q$-Steiner systems won't exist too. On the other hand, the existence of such a design
might lead to a construction for the related $q$-Steiner system.

For puncturing there is an inverse operation called \emph{extension}.
A related operation for our discussion, which is completely of a different nature from extension, is
the \emph{expansion}. These two operations are defined next.

A $t$-subspace $X$ of $\F_q^m$ is \emph{extended} to a $t'$-subspace $Y$ of $\F_q^{m'}$, where
$t' \geq t$, $m'>m$, and $m'-m \geq t'-t$, if $X$ is the subspace obtained from $Y$ by
puncturing $Y$ $m'-m$ times. Note, that the \emph{extension} of a subspace is not always unique, as we will
prove in the sequel (see lemma~\ref{lem:addNone}), but there is always a unique outcome for puncturing. In other words, a
$p$-punctured $t$-subspace $X$ of $\F_q^m$
can be obtained from a few different $t'$-subspaces of $\F_q^{m+p}$.
The new columns are added as the last columns of extended subspace. But, similarly to
puncturing, they can added theoretically between any set of columns.
To make the paper consistent, we will make the extensions only to the end of the
columns of the related subspaces, unless otherwise is specifically stated.

To prove our claims we will need some general way to represent subspaces,
in such a way that the representation of the punctured subspace $X'$ will be derived directly
from the representation of the subspace $X$. For this purpose two different
representations will be used for a $k$-subspace $X$ of $\F_q^m$.
The first representation is by an $(q^k-1) \times m$ matrix which contains
the $q^k-1$ nonzero vectors of $X$. Each nonzero vector of $X$ is a row in this matrix.
The second representation is by
a $k \times m$ matrix which is the generator matrix for $X$ in reduced row echelon form.
A $k \times n$ matrix with rank $k$ is in {\it reduced row echelon
form} if the following conditions are satisfied.
\begin{itemize}
\item The leading coefficient of a row is always to the right of
the leading coefficient of the previous row.

\item All leading coefficients are {\it ones}.

\item Every leading coefficient is the only nonzero entry in its
column.
\end{itemize}

The next definition of virtual subspace is essential in understanding our exposition.
An $r$-subspace $Y$ of $\F_q^m$ is called a \emph{virtual} $k$-\emph{subspace} of $\F_q^m$,
$r \leq k$, if $Y$ was punctured from a $k$-subspace $X$
of $\F_q^n$, $n>m$ and it is represented by a $(q^k-1) \times m$ matrix which
represents the actual outcome when $Y$ was punctured. In other words the representation of $Y$ is
obtained by the deleting the last $n-m$ columns from the $(q^k-1) \times n$ matrix which
represents $X$. It is important to understand that $Y$ has exactly one representation
as a virtual $k$-subspace, no matter from which subspace it was punctured.
Using this representation, $Y$ has $\sbinomq{k}{t}$ $t$-subspaces for each
$0 \leq t \leq k$, some of them are identical and some of them are virtual (see Example~\ref{ex:virtual}).
This will be used later in our enumerations.
Note, that with the virtual $k$-subspace of the punctured subspace, it is easier to see how
the $t$-subspaces were punctured. This is easily seen in Example~\ref{ex:virtual}.
Finally, note that a $k$-subspace of $\F_q^m$ is always also a virtual $k$-subspace
of $\F_q^m$, where both have the same matrix representation.

An $r$-subspace $X$ of $\F_q^m$ is \emph{expanded} to a (virtual) $k$-subspace $Y$ of $\F_q^m$, $r \leq k$, if
$X$ can be obtained by puncturing a $k$-subspace $Z$, and $Y$ is the representation of $X$
by a $(q^k-1) \times m$ matrix after the
puncturing. The virtual $k$-subspace $Y$ is also called the $k$-\emph{expansion} of $X$.


\begin{lemma}[\bf from a $t$-subspace to a $t$-subspace, one extension]
\label{lem:addNone}
If $X$ is a $t$-subspace of $\F_q^m$, then it can be extended
in exactly $q^t$ distinct ways to a $t$-subspaces of $\F_q^{m+1}$.
\end{lemma}
\begin{proof}
Any one of the $q^t$ distinct linear combinations of the columns of $X$ can be appended as the new $(m+1)$th column and
each such linear combination, appended as the $(m+1)$th column, yields a different $t$-subspace of $\F_q^{m+1}$.
\end{proof}

\begin{example}
Let $X$ be the following 2-subspace of $\F_2^4$ represented by the $3 \times 4$ matrix,
$$
X=\begin{array}{c}
0100 \\
0010 \\
0110
\end{array}~,
$$
It can be extended in exactly 4 distinct ways to 2-subspaces of $\F_2^5$, represented by the $3 \times 5$ matrices,
$$
\begin{array}{c}
01000 \\
00100 \\
01100
\end{array}, ~~
\begin{array}{c}
01001 \\
00100 \\
01101
\end{array},~~
\begin{array}{c}
01000 \\
00101 \\
01101
\end{array},~~
\begin{array}{c}
01001 \\
00101 \\
01100
\end{array} ~.
$$
\end{example}

\begin{lemma}[\bf from $t$-subspace to $(t+1)$-subspace, one extension]
\label{lem:addOne}
If $X$ is a $t$-subspace of $\F_q^m$, then it can be extended
in exactly one way to a $(t+1)$-subspace of $\F_q^{m+1}$.
\end{lemma}
\begin{proof}
Any extension which increase the dimension by one is equivalent to first adding a column of
zeroes to the existing $t$-subspace $X$ to form the $t$-subspace $\hat{X}$ of $\F_q^{m+1}$.
Then the only extension to a $(t+1)$-subspace is $\Span{\hat{X},e_{m+1}}$, where $\Span{Z}$
denotes the linear span of~$Z$.
\end{proof}

\begin{example}
\label{ex:virtual}
Let $X$ be the following 2-subspace of $\F_2^4$ represented by the $3 \times 4$ matrix,
$$
X=\begin{array}{c}
0100 \\
0010 \\
0110
\end{array} ~,
$$
and let
$$
Y=\begin{array}{c}
01000 \\
00100 \\
01100
\end{array}~,
$$
be one of its extensions to a 2-subspace of $\F_2^5$ represented by the $3 \times 5$ matrix.
$X$ can be extended in a unique way to a 3-subspace of $\F_2^5$, and $Y$ has a unique representation as
a virtual 3-subspace (up to permutations of rows). They are represented by the two $7 \times 5$ matrices, respectively,
$$
\begin{array}{c}
01000 \\
00100 \\
01100 \\
01001 \\
00101 \\
01101 \\
00001
\end{array}, ~~
\begin{array}{c}
01000 \\
00100 \\
01100 \\
01000 \\
00100 \\
01100 \\
00000
\end{array}~.
$$
The extended 3-subspace contains the following seven 2-subspaces
$$
\begin{array}{c}
01000 \\
00100 \\
01100
\end{array}, ~~
\begin{array}{c}
01000 \\
00101 \\
01101
\end{array},~~
\begin{array}{c}
00100 \\
01001 \\
01101
\end{array},~~
\begin{array}{c}
01100 \\
01001 \\
00101
\end{array},~~
\begin{array}{c}
01000 \\
01001 \\
00001
\end{array},~~
\begin{array}{c}
00100 \\
00101 \\
00001
\end{array},~~
\begin{array}{c}
01100 \\
01101 \\
00001
\end{array},
$$
while the virtual 3-subspace contains the following related seven 2-subspaces, from which three are virtual
(note that the distinction between the two sets of seven subspaces is the last column),
$$
\begin{array}{c}
01000 \\
00100 \\
01100
\end{array}, ~~
\begin{array}{c}
01000 \\
00100 \\
01100
\end{array},~~
\begin{array}{c}
00100 \\
01000 \\
01100
\end{array},~~
\begin{array}{c}
01100 \\
01000 \\
00100
\end{array},~~
\begin{array}{c}
01000 \\
01000 \\
00000
\end{array},~~
\begin{array}{c}
00100 \\
00100 \\
00000
\end{array},~~
\begin{array}{c}
01100 \\
01100 \\
00000
\end{array}~.
$$
\end{example}

Finally, in this section we consider the outcome of puncturing a $q$-Steiner systems
and some properties of the related puncturing. For this purpose we need to use
the following well-known equation~\cite[p. 444]{McSl77}.

\begin{lemma}
\label{eq:gauss1}
If $1 \leq k \leq n-1$, then $\sbinomq{n}{k} = q^k \sbinomq{n-1}{k} + \sbinomq{n-1}{k-1}$, where
$\sbinomq{n}{0} = \sbinomq{n}{n} =1$.
\end{lemma}

\begin{theorem}
\label{thm:Sd1}
If $\dS$ is a $q$-Steiner system $\dS_q(t,k,n)$, then the punctured system $\dS'$
contain exactly $\frac{\sbinomq{n-1}{t-1}}{\sbinomq{k-1}{t-1}}$ distinct $(k-1)$-subspaces
which form a $q$-Steiner system $\dS_q(t-1,k-1,n-1)$, denoted by $\tilde{\dS}$. Each $t$-subspace of $\F_q^{n-1}$
which is contained in a $(k-1)$-subspace of $\tilde{\dS}$ is not contained
in any of the $k$-subspaces of $\dS'$.
Each $t$-subspace of $\F_q^{n-1}$ which is not contained in a $(k-1)$-subspace of $\tilde{\dS}$,
appears exactly $q^t$ times in the $k$-subspaces of $\dS'$.
\end{theorem}
\begin{proof}
By Lemma~\ref{lem:1-punctured}, $\dS'$ can contain only $(k-1)$-subspaces and $k$-subspaces.
A~$k$-subspace which contains the unity vector with a \emph{one} in the last coordinate
is punctured into a $(k-1)$-subspace, while other $k$-subspaces are punctured into $k$-subspaces.

Let $X$ be a $(t-1)$-subspace of $\F_q^{n-1}$. By Lemma~\ref{lem:addOne}, there is a unique way
to extend it to a $t$-subspace $Y$ of $\F_q^n$. Therefore, since $\dS$ is a $q$-Steiner system $\dS_q(t,k,n)$, it follows
that $Y$ is contained in exactly one $k$-subspace $Z$ of $\dS$. Clearly, $Z$ is a $k$-subspace
in $\F_q^n$ which contains the unity vector with a \emph{one} in the last coordinate.
Hence, $Z$ is punctured into a $(k-1)$-subspace $Z'$ which is the only $(k-1)$-subspace
of $\dS'$ containing $X$. Therefore, all the $k$-subspaces of $\dS$ which contain
the unity vector with a \emph{one} in the last coordinate are punctured into
a $q$-Steiner system $\dS_q(t-1,k-1,n-1)$.

Let $Z$ be a $k$-subspace of $\F_q^n$ which contains $\Span{e_n}$. The
virtual $k$-subspace, obtained from the punctured $(k-1)$-subspace $Z'$,
contains $\sbinomq{k}{t}$ $t$-subspaces, some of them are identical
and some of them are virtual $t$-subspaces (see Example~\ref{ex:virtual}).
Since $\Span{e_n} \subset Z$ it follows that $Z$ has $\sbinomq{k-1}{t-1}$ $t$-subspaces
which contains $\Span{e_n}$ and hence $\sbinomq{k}{t}-\sbinomq{k-1}{t-1}$ $t$-subspaces
which do not contain $\Span{e_n}$. These are the only $t$-subspaces contained in $Z$.
Since $\Span{e_n} \subset Z$ it follows that $Z'$ is a $(k-1)$-subspace and hence it has
$\sbinomq{k-1}{t}$ distinct $t$-subspaces. The $\sbinomq{k}{t}-\sbinom{k-1}{t-1}$ $t$-subspaces
of $Z$ which do not contain $\Span{e_n}$ are punctured, and each $(t-1)$-subspace obtained in
this way is obtained the same amount of times in the set of punctured $(t-1)$-subspaces.
Therefore, each such $(t-1)$-subspace is contained $\frac{\sbinomq{k}{t}-\sbinomq{k-1}{t-1}}{\sbinomq{k-1}{t}}$
times in the punctured $t$-subspaces of $\F_q^{n-1}$ obtained from $\dS$.
By Lemma~\ref{eq:gauss1} we have $\frac{\sbinomq{k}{t}-\sbinomq{k-1}{t-1}}{\sbinomq{k-1}{t}}=q^t$
(by Lemma~\ref{lem:addNone} these are exactly all their appearances and hence no one could have appeared more times).

By Lemma~\ref{lem:addNone}, each $t$-subspace of $\F_q^{n-1}$ can be extended in $q^t$
distinct ways to a $t$-subspace of $\F_q^n$.
This implies, that each $t$-subspace which is not contained in $\tilde{\dS}$,
appears exactly $q^t$ times in the other $k$-subspaces of $\dS'$. Simple counting shows
that we have covered all the subspaces of the punctured system,
which completes the proof of the theorem.
\end{proof}
We note, that the $q$-Steiner system $\dS_q(t-1,k-1,n-1)$ of Theorem~\ref{thm:Sd1}
is the same as the one constructed in~\cite{ScEt}, but the extra factor of the Theorem are
the $k$-subspaces which do not belong to the $q$-Steiner system.
\begin{cor}
\label{cor:Sd1}
If $\dS$ is a $q$-Steiner system $\dS_q(k-1,k,n)$, then the punctured system $\dS'$
has a set $\tilde{\dS}$ with $\frac{\sbinomq{n-1}{k-2}}{\sbinomq{k-1}{k-2}}$ different $(k-1)$-subspaces
which form a $q$-Steiner system $\dS_q(k-2,k-1,n-1)$. Each other $(k-1)$-subspace
which is not contained in $\tilde{\dS}$, appears exactly $q^{k-1}$ times in the $k$-subspaces of $\dS'$.
\end{cor}

The first goal of this paper is to define a new type of design $\dS_q(t,k,n;m)$ which
contains the possible subspaces punctured from a $q$-Steiner system $\dS_q(t,k,n)$.

\begin{defn}
A $p$-\emph{punctured} $q$-\emph{Steiner system} $\dS_q(t,k,n;m)$, $m=n-p$, is a multi-set $\dS$ of subspaces of $\F_q^m$,
in which each $t$-subspace of $\F_q^n$ can be obtained exactly once by extending
$p$ times all the subspaces of $\dS$, where the appearances of the same
subspace of $\F_q^m$ in $\dS$ are extended together (in parallel) for this purpose.
The appearances of distinct subspaces of $\F_q^m$ in $\dS$ are extended in a sequence,
where the order of the different subspaces of $\dS$ in this sequence is arbitrary.
\end{defn}
An equivalent definition, and maybe easier to understand, can be given in terms of virtual subspaces.
\begin{defn}
A $p$-\emph{punctured} $q$-\emph{Steiner system} $\dS_q(t,k,n;m)$, $m=n-p$,
is a multi-set $\dS$ of subspaces of $\F_q^m$, satisfying the following two requirements.

\begin{enumerate}
\item The number of subspaces in $\dS$ is the same as the number of subspaces in
a $q$-Steiner system $\dS_q(t,k,n)$.

\item Let $\tilde{\dS}$ be a system which contains the virtual $k$-subspaces
of the subspaces in $\dS$ ($\dS$ and $\tilde{\dS}$ have the same size).
Let $\T$ be the set of all $t$-subspaces of $\F_q^n$ and $\T'$ be the set
of all $p$-punctured $t$-subspaces of $\F_q^m$. For each subspace $X \in \T'$ let
$$\lambda (X) = |  \{ Y ~:~ Y \in \T , ~ X ~ \text{is a}~ p\text{-punctured}~ t\text{-subspace of} ~ Y  \} |~.$$
It is required that for each $X \in \T'$, $X$ will be appear $\lambda (X)$ times as a virtual
$t$-subspace in the virtual $k$-subspaces of $\tilde{\dS}$.
\end{enumerate}
\end{defn}

\begin{example}
Let $\dS$ be a systems which consists of 336 1-subspaces and 45 0-subspaces of $\F_2^1$.
There are $\sbinomtwo{6}{2} =651$ 2-subspaces of $\F_2^7$ whose first column is the all-zero column.
Each extension of an 1-subspace of $\F_2^1$ will contribute one 2-subspaces of $\F_2^7$
whose first column is all-zero, while
each extension of a 0-subspace of $\F_2^1$ will contribute seven 2-subspaces of $\F_2^7$
whose first column is all-zero. Hence, the extension in parallel produces $336+45 \cdot 7 =651$
such subspaces as required. The same goes for the other 2-subspaces and hence $\dS$ is
a 6-punctured $q$-Steiner system $\dS_2(2,3,7;1)$.
\end{example}

As an immediate trivial result is the following lemma.
\begin{lemma}
\label{lem:punct_punct}
If there exists a $p$-punctured $q$-Steiner system $\dS_q(t,k,n;m)$, then
there exists a $(p+1)$-punctured $q$-Steiner system $\dS_q(t,k,n;m-1)$.
\end{lemma}

The following theorem is given for a punctured $q$-Steiner system $\dS_q(t,k,n;n-1)$.
It can be generalized to other $p$-punctured $q$-Steiner systems, $p>1$. For simplicity
and since the case $p=1$ is the most informative
we prove only this case.

In several proofs we will need to use concatenation of two matrices. Let $X_1$ and $X_2$
be two $\ell \times m_1$ and $\ell \times m_2$ matrices, respectively. The
\emph{concatenation} of $X_1$ and $X_2$, $X_1 \circ X_2$ ($X_1$ or $X_2$ can be also columns)
is an $\ell \times (m_1 + m_2)$ matrix whose first $m_1$ columns is $X_1$ and its
last $m_2$ columns is $X_2$.

\begin{theorem}
\label{thm:diff_punc}
If $\dS'$, a punctured $q$-Steiner system $\dS_q(t,k,n;n-1)$, $1 < t < k <n$, was
obtained by puncturing a $q$-Steiner system $\dS_q(t,k,n)$, then all subspaces
of $\dS'$ are distinct.
\end{theorem}
\begin{proof}
Let $\dS$ be a $q$-Steiner system $\dS_q(t,k,n)$.
By Theorem~\ref{thm:Sd1}, all the $(k-1)$-subspaces of $\dS'$,
the punctured $q$-Steiner systems $\dS_q(t,k,n;n-1)$,
are distinct since they form a $q$-Steiner system $\dS_q(t-1,k-1,n-1)$.
Hence, we only have to prove that there are no two
equal $k$-subspaces in $\dS'$.

Assume that $X$ and $Y$ are two distinct $k$-subspaces which are punctured to
$k$-subspaces $X'$ and $Y'$. Assume that $X'=Y'$ and consider their representation
by $(q^k-1) \times (n-1)$ matrices. Let $x$ and $y$ be the last two columns of
$X$ and $Y$, respectively.
Let $x\circ y$ denote the $(q^k-1) \times 2$ matrix formed by concatenating
$x$ and $y$ in this order.
Clearly, $x \neq y$ since otherwise $X=Y$. Neither $x$
nor $y$ can be an all-zero column since otherwise both $X$ and $Y$ will contain
the same $(k-1)$-subspace (relates to the zeroes in the nonzero column).
Hence, $x \circ y$ is
the $k$-expansion of a 2-subspace in $\F_q^2$. This 2-subspace contains all the vectors
of $\F_q^2$ including the ones in which the two elements are nonzero and equal. These
row vectors in $x \circ y$ with the rows of \emph{zeroes} (clearly contained in the $k$-expansion
of a 2-subspace of $\F_q^2$ since $k>2$) form a $(k-1)$-expansion of
a 1-subspace. Therefore, since $k-1 \geq t$ we have that $X$ and $Y$ contain one common
$t$-subspace whose last ($n$th) column is the corresponding $t$-expansion of this 1-subspace,
a contradiction to the fact that $\dS$ is a $q$-Steiner systems $\dS_q(t,k,n)$.
Thus, $X' \neq Y'$ and the proof of the theorem is completed.
\end{proof}

\section{System of Equations}
\label{sec:system}

Corollary~\ref{cor:ness} yields a set of necessary conditions for the existence of a $q$-Steiner system $\dS_q(t,k,n)$.
Similar necessary conditions can be derived for any $p$-punctured
$q$-Steiner system. Some of these conditions yield new interesting equations which must be satisfied.
In this section we will derive these new necessary conditions.

Let $\tilde{\dS}$ be an $(n-m)$-punctured $q$-Steiner system $\dS_q(t,k,n;m)$ and let $p=n-m$, i.e. $\dS_q(t,k,n;m)$
is a $p$-punctured $q$-Steiner system. We start with two simple lemmas which are implied by our
previous discussion on punctured subspaces. The first lemma is an immediate
consequence of Lemma~\ref{lem:1-punctured}.

\begin{lemma}[\bf dimension of subspaces to be covered]
\label{lem:s_dim}
Suppose $\dS$ is a $q$-Steiner system $\dS_q(t,k,n)$. Let $\T$ be the set of $t$-subspaces
which are covered by the blocks of $\dS$. Then in the $p$-punctured $q$-Steiner
system $\dS_q(t,k,n;m)$, $p=n-m$, derived from $\dS$, each element of $\T$ corresponds to
a $p$-punctured $s$-subspace, where $\max \{0,t-p\} \leq s \leq \min \{t,m\}$.
\end{lemma}

\begin{lemma}[\bf dimension of subspaces which cover the $p$-punctured $s$-subspaces]
\label{lem:r_dim}
Suppose $\dS$ is a $q$-Steiner system $\dS_q(t,k,n)$.
Let $\T$ be the $t$-subspaces which are covered by the blocks of $\dS$.
Then in the $p$-punctured $q$-Steiner
system $\dS_q(t,k,n;m)$, $p=n-m$, derived from $\dS$, each element of $\T$ corresponds to
a $p$-punctured $s$-subspace, covered by a $p$-punctured $r$-subspace,
where $\max \{ k-p,s \} \leq r \leq \min \{k-t+s,m\}$.
\end{lemma}
\begin{proof}
The lower bound is a consequence of Lemma~\ref{lem:1-punctured} and the fact
that an $s$-subspace cannot be covered by a subspace of a smaller dimension.
Since the subspaces of $\dS_q(t,k,n;m)$ are subspaces of $\F_q^m$, it follows that
$r \leq m$. Finally, if a $t$-subspace $X$ was punctured to an $s$-subspace then
the $k$-subspace $Y$ which covers $X$ must also be reduced by at least $t-s$ times
in its dimension and hence $r \leq k-t+s$.
\end{proof}

We are now in a position to describe a set of equations, related to the
$p$-punctured $q$-Steiner system $\dS_q(t,k,n;m)$, which must be satisfied
if the $p$-punctured $q$-Steiner system $\dS_q(t,k,n;m)$ exists.
Each $s$-subspace $X$ of $\F_q^m$, $\max \{0,t-p\} \leq s \leq \min \{t,m\}$, yields
one equation related to the way it is covered by $\dS_q(t,k,n;m)$.
Each $r$-subspace $Y$ of $\F_q^m$, $\max \{ k-p,0 \} \leq r \leq \min \{k,m\}$,
yields one nonnegative integer variable, $a_Y$, which is the number of appearances of $Y$ in
the $p$-punctured $q$-Steiner system $\dS_q(t,k,n;m)$. In the equation
for the $s$-subspace $X$ we have a linear combination of the variables for the $r$-subspaces
of $\F_q^m$ which contain $X$.

\begin{example}
Assume that we want to examine the 5-punctured $q$-Steiner system $\dS_2(2,3,7;2)$.
Clearly, by Lemma~\ref{lem:s_dim}, the 2-subspaces of $\F_2^7$ were punctured to $s$-subspaces of $\F_2^2$,
where $0 \leq s \leq 2$. There is exactly one 0-subspace, three 1-subspaces, and one
2-subspaces, of $\F_2^2$, represented as virtual 2-subspaces by the following five
$3 \times 2$ matrices.
$$
\begin{array}{c}
00 \\
00 \\
00
\end{array}, ~~
\begin{array}{c}
10 \\
10 \\
00
\end{array},~~
\begin{array}{c}
01 \\
01 \\
00
\end{array},~~
\begin{array}{c}
11 \\
11 \\
00
\end{array},~~
\begin{array}{c}
01 \\
10 \\
11
\end{array}~.
$$
Clearly, by Lemma~\ref{lem:r_dim} these s-subspaces of $\F_2^7$ are covered by $r$-subspaces of $\F_2^2$,
where $0 \leq r \leq 2$. There is exactly one 0-subspace, three 1-subspaces, and one
2-subspaces, of $\F_2^2$, represented as virtual 3-subspaces by the following five
$7 \times 2$ matrices.
$$
X=\begin{array}{c}
00 \\
00 \\
00 \\
00 \\
00 \\
00 \\
00
\end{array}, ~~
Y=\begin{array}{c}
10 \\
10 \\
10 \\
10 \\
00 \\
00 \\
00
\end{array},~~
Z=\begin{array}{c}
01 \\
01 \\
01 \\
01 \\
00 \\
00 \\
00
\end{array},~~
U=\begin{array}{c}
11 \\
11 \\
11 \\
11 \\
00 \\
00 \\
00
\end{array},~~
V=\begin{array}{c}
01 \\
01 \\
10 \\
10 \\
11 \\
11 \\
00
\end{array}~.
$$
Therefore, we have 5 variables $a_X$, $a_Y$, $a_Z$, $a_U$, and $a_V$. The system of equations
consists of the following five equations:
$$
\begin{array}{c}
155 = 7 \cdot a_X +a_Y + a_Z + a_U \\
496 = 6 \cdot a_Y + a_V \\
496 = 6 \cdot a_Z + a_V \\
496 = 6 \cdot a_U + a_V \\
1024 = 4 \cdot a_V
\end{array}~.
$$
The first equation is constructed as follows: there are $\sbinomtwo{5}{2}=155$
2-subspaces whose first two columns are \emph{zeroes}. For all the seven 2-subspaces
resulting from the virtual 3-subspace $X$ the first two columns are \emph{zeroes}. Only for one such 2-subspaces
resulting from $Y$, $Z$, or $U$ the first two columns are \emph{zeroes}.
This explains the first equation. The other four equations are constructed in a similar way.
There is a unique solution for this set of equations, $a_X=5$, $a_Y=a_Z=a_U=40$, and $a_V=256$,
and hence the 5-punctured $q$-Steiner system $\dS_2(2,3,7;2)$ exists.
\end{example}

The variables
which appear in each equation and their coefficients in the equation
should be computed in advance as will be done next.
First we have to compute the number of $t$-subspaces in $\F_q^n$
which are formed by extending an $s$-subspace $X$ of $\F_q^m$.
Let $N_{(s,m),(t,n)}$ be the number of distinct $t$-subspaces in
$\F_q^n$ which are formed by extending a given $s$-subspace $X$ of $\F_q^m$.

\begin{example}
Let $X$ be the following 2-subspace of $\F_2^5$ represented by the $3 \times 5$ matrix,
$$
X=\begin{array}{c}
01001 \\
00101 \\
01100
\end{array} ~.
$$
It can be extended to the following $N_{(2,5),(3,7)}=12$ 3-subspaces of $\F_2^7$ represented by the $7 \times 7$ matrices,
$$
\begin{array}{c}
0100100 \\
0010100 \\
0110000 \\
0100101 \\
0010101 \\
0110001 \\
0000001
\end{array},~~
\begin{array}{c}
0100110 \\
0010110 \\
0110000 \\
0100111 \\
0010111 \\
0110001 \\
0000001
\end{array},~~
\begin{array}{c}
0100110 \\
0010100 \\
0110010 \\
0100111 \\
0010101 \\
0110011 \\
0000001
\end{array},~~
\begin{array}{c}
0100100 \\
0010110 \\
0110010 \\
0100101 \\
0010111 \\
0110011 \\
0000001
\end{array},~~
\begin{array}{c}
0100100 \\
0010100 \\
0110000 \\
0100110 \\
0010110 \\
0110010 \\
0000010
\end{array},~~
\begin{array}{c}
0100101 \\
0010100 \\
0110001 \\
0100111 \\
0010110 \\
0110011 \\
0000010
\end{array},~~
$$
$$
\begin{array}{c}
0100100 \\
0010101 \\
0110001 \\
0100110 \\
0010111 \\
0110011 \\
0000010
\end{array},~~
\begin{array}{c}
0100101 \\
0010101 \\
0110000 \\
0100111 \\
0010111 \\
0110010 \\
0000010
\end{array},~~
\begin{array}{c}
0100100 \\
0010100 \\
0110000 \\
0100111 \\
0010111 \\
0110011 \\
0000011
\end{array},~~
\begin{array}{c}
0100101 \\
0010100 \\
0110001 \\
0100110 \\
0010111 \\
0110010 \\
0000011
\end{array},~~
\begin{array}{c}
0100100 \\
0010101 \\
0110001 \\
0100111 \\
0010110 \\
0110010 \\
0000011
\end{array},~~
\begin{array}{c}
0100101 \\
0010101 \\
0110000 \\
0100110 \\
0010110 \\
0110011 \\
0000011
\end{array}~.
$$
\end{example}

\begin{lemma}
If $0<m<n$ and $0 \leq s \leq t$, then $N_{(s,m),(t,n)}=q^{s(n-m-t+s)} \sbinomq{n-m}{t-s}$.
\end{lemma}
\begin{proof}
The $s$-subspace $X$ of $\F_q^m$ is represented by an $s \times m$ matrix $G_1$ in reduced row echelon form.
A $t$-subspace $Y$ formed by extending $X$ to a $t$-subspace
of $\F_q^n$ is represented by a $t\times n$ generator
matrix $G$ in reduced row echelon form. The upper left $s \times m$ matrix of $G$ is the
generator matrix $G_1$ of $X$ and hence $G$ has the following structure.
\begin{large}
\begin{align*}
\left[ \begin{array}{cc}
G_1 & B\\
{\bf 0} & G_2
\end{array}
\right] .
\end{align*}
\end{large}
The new $n-m$ columns (in $G$ relatively to $G_1$), restricted
to the last $t-s$ rows, forms a generator matrix for a $(t-s)$-subspace of $\F_q^{n-m}$. This
$(t-s) \times (n-m)$ generator matrix $G_2$ is in reduced row echelon form, where the
columns with leading \emph{ones} are the ones in which the dimension is
increased during the extension (see Lemma~\ref{lem:addOne}). This generator matrix can be chosen in
$\sbinomq{n-m}{t-s}$ distinct ways since it forms a $(t-s)$-subspace of $\F_q^{n-m}$.
In the other new $n-m-t+s$ columns of $G$ (columns with no leading \emph{ones})
the dimension is not increased compared to the original $s$-subspace $X$ and
hence they can be chosen (after the columns with the leading \emph{ones} of $G_2$
are fixed; see Lemma~\ref{lem:addNone}) in $q^{s(n-m-t+s)}$ distinct ways. The reason is that $B$ has
$s$ rows and $n-m$ columns from which the entries of $n-m-t+s$ columns can be chosen arbitrarily from $\F_q$.
It leads to a total of $q^{s(n-m-t+s)} \sbinomq{n-m}{t-s}$ distinct ways to form this extension.
Note, that the columns of $G_2$ do not (and need not) contribute to these extensions.
Each $(t-s)$-subspace of $\F_q^{n-m}$ is combined
with the related $s$-subspace in $\F_q^m$ to form a $t$-subspace of $\F_q^n$.
\end{proof}

For a given $s$-subspace $X$ of $\F_q^m$, $N_{(s,m),(t,n)}$ should be equal to the number
of $r$-subspaces in the $p$-punctured $q$-Steiner system $\dS_q(t,k,n;m)$,
$\max \{ k-p,s \} \leq r \leq \min \{k-t+s,m\}$, which contains $X$.
Note, that if $r<k$, then there are $r$-subspaces in $\dS_q(t,k,n;m)$ which contain $X$
more than once, since we should look on the $k$-expansion of the $r$-subspace.
Clearly, for this purpose we also have to consider
the $t$-expansion of the related $s$-subspace~$X$.
Let $C_{(s,t),(r,k)}$ be the number of copies of the $t$-expansion $\tilde{X}$ obtained from the $s$-subspace $X$
in $\F_q^m$ (note that $X$ is $t$-expanded in a unique way),
which are contained in the $k$-expansion $\tilde{Y}$ of an $r$-subspace $Y$ in $\F_q^m$,
such that $X$ is a subspace of $Y$.

\begin{example}
Let $X$ be the following 2-subspace of $\F_2^4$ represented by the $3 \times 4$ matrix,
$$
X=\begin{array}{c}
0100 \\
0010 \\
0110
\end{array} ~.
$$
If $Y=X$, then the 3-expansion $\tilde{Y}$ of $Y$ is represented by $7 \times 4$ matrix,
and one of its extensions $\hat{Y}$ for a 3-subspace is represented by a $7 \times 7$ matrix as follows:
$$
\tilde{Y}=\begin{array}{c}
0100 \\
0010 \\
0110 \\
0100 \\
0010 \\
0110 \\
0000
\end{array} ,~~~
\hat{Y}=\begin{array}{c}
0100000 \\
0010000 \\
0110000 \\
0100101 \\
0010101 \\
0110101 \\
0000101
\end{array} ~.
$$
$\hat{Y}$ contains $C_{(2,2),(2,3)}=4$ 2-subspaces extended from $X$ represented by the $3 \times 7$ matrices,
$$
\begin{array}{c}
0100000 \\
0010000 \\
0110000
\end{array} ,~~
\begin{array}{c}
0100000 \\
0010101 \\
0110101
\end{array} ,~~
\begin{array}{c}
0100101 \\
0010000 \\
0110101
\end{array} ,~~
\begin{array}{c}
0100101 \\
0010101 \\
0110000
\end{array} ~.
$$
\end{example}

\begin{lemma}
If $0 \leq s \leq t <k$ and $s \leq r \leq k-t+s$, then $C_{(s,t),(r,k)}=\sbinomq{k-r}{t-s} q^{s(k-r-t+s)}$.
\end{lemma}
\begin{proof}
An $r$-subspace $Y$ of $\F_q^m$ can be represented by a $(q^r-1) \times m$ matrix
whose rows are the nonzero vectors of $Y$. This $r$-subspace $Y$ is $k$-expanded in $\F_q^m$
by writing in a $(q^k-1)\times m$ matrix vertically $q^{k-r}$ copies of $Y$ and after them
$(q^{k-r} -1) \times m$ all-zero matrix $Z$. This forms a $(q^k-1) \times m$ matrix
which represents the $k$-expansion of $Y$.
If $Y$ was $p$-punctured from a $k$-subspace $W$ of $\F_q^n$, then $W$ is formed from
the $k$-expansion of $Y$ by concatenating to $Z$ a $(q^{k-r} -1) \times (n-m)$
matrix $\tilde{Z}$ which represent a $(k-r)$-subspace of $\F_q^{n-m}$.
The $k$-subspace of $\F_q^n$ is a direct sum of this
$(k-r)$-subspace in $\F_q^n$ with an extension of $Y$ to an $r$-subspace in~$\F_q^n$,

Similarly, the $s$-subspace $X$, which
is a subspace of $Y$, is extended and expanded to a $(k-r+s)$-subspace by writing in a
$(q^k-1)\times m$ matrix vertically $q^{k-r}$
copies of $X$ and after them $(q^{k-r} -1) \times m$ all-zero matrix~$Z$. Each $t$-subspace in $\F_q^n$ which
is extended from $X$, in the $q$-Steiner system $\dS_q(t,k,n)$, is constructed by first choosing
a $(t-s)$-subspace from $\tilde{Z}$, which can be done in $\sbinomq{k-r}{t-s}$
different ways. The $(t-s)$-subspace is completed to a $t$-subspace by performing
direct sum with the extension of $X$. The $s$-subspace $X$ can be chosen in a few distinct ways
from the $k$-expansion of $Y$. Each vector from a given
basis of $X$ can be chosen in $q^{k-r}$ distinct ways (since $q^{k-r}$ copies of $Y$
were written). But, since each vector of $X$ appears
$q^{t-s}$ times in the $t$-expansion of $X$, it follows that each choice of $X$ is chosen
in $(\frac{q^{k-r}}{q^{t-s}})^s$ distinct ways. This implies that
$C_{(s,t),(r,k)}=\sbinomq{k-r}{t-s} (\frac{q^{k-r}}{q^{t-s}})^s=\sbinomq{k-r}{t-s} q^{s(k-r-t+s)}$
\end{proof}

Now, for a given $s$-subspace $X$ of $\F_q^m$, $N_{(s,m),(t,n)}$ should be equal
to the sum over all $r$-subspaces which contain $X$, where $C_{(s,t),(r,k)}$, for a given $r$-subspace $Y$, is multiplied
by $a_Y$ (see the definitions after Lemma~\ref{lem:r_dim}),
the number of appearance of $Y$ in the $p$-punctured $q$-Steiner system $\dS_q(t,k,n;m)$.
For a given $r$, $\max \{ k-p,s \} \leq r \leq \min \{k-t+s,m\}$, let
$D_{s,r,m}$ be the number of $r$-subspaces which contain a given
$s$-subspace in $\F_q^m$. In other words, $D_{s,r,m}$ is the number of variables for $r$-subspaces
which appear in the equation for any given $s$-subspace.

\begin{lemma}
If $0 \leq s \leq r \leq m$, then $D_{s,r,m}= \sbinomq{m-s}{r-s}$.
\end{lemma}
\begin{proof}
Let $X$ be an $s$-subspace of $\F_q^m$. We enumerate the number of distinct
$r$-subspaces which contain $X$ in $\F_q^m$, by adding linearly independent vectors one by one to $X$.
The first vector can be chosen in $q^m-q^s$ distinct ways, the second in $q^m-q^{s+1}$ distinct ways and the
last in $q^m - q^{r-1}$ different ways, for a total of $\Pi_{i=r-1}^s (q^m-q^i)$ different ways.
Similarly, a given $r$-subspace $Y$ which is formed in this way can be constructed in
$\Pi_{i=r-1}^s (q^r-q^i)$ distinct ways (the first vector can be chosen in $q^r-q^s$ distinct ways and so on).
Hence, the total number of distinct $r$-subspaces formed in this way is
$$\frac{\prod_{i=r-1}^s (q^m-q^i)}{\prod_{i=r-1}^s (q^r-q^i)}=\prod_{i=r-1}^s \frac{q^m-q^i}{q^r-q^i}
=\prod_{i=r-1}^s \frac{q^{m-i}-1}{q^{r-i}-1}=\sbinomq{m-s}{m-r}=\sbinomq{m-s}{r-s}$$
\end{proof}

So far, we have described the computation of the components in the
equations that should be satisfied if a $p$-punctured $q$-Steiner system $\dS_q(t,k,n;m)$ exists.
The solution for the variables must be nonnegative integers. Before we describe the specific
equations, and before we reduce the number of equations
in some cases, we compute the total number of equations and the total number
of variables in the equations for the $p$-punctured $q$-Steiner system $\dS_q(t,k,n;m)$.

\begin{lemma}
\label{lem:num_eq}
The number of equations for the $p$-punctured $q$-Steiner system $\dS_q(t,k,n;m)$, $m=n-p$, is
$$
\sum_{s= \max \{0,t-p\}}^{\min\{t,m\}} \sbinomq{m}{s} ~,
$$
i.e. $\sbinomq{m}{s}$ equations for all the $s$-subspaces of $\F_q^m$,
where $\max \{0,t-p\} \leq s \leq \min \{t,m\}$.
\end{lemma}
\begin{proof}
The range of $s$ is a direct consequence from Lemma~\ref{lem:s_dim}. For each
$s$-subspaces of $\F_q^m$ we have one equation and hence there are $\sbinomq{m}{s}$
equations for each $s$.
\end{proof}

\begin{lemma}
\label{lem:num_var}
The number of variables for the $p$-punctured $q$-Steiner system $\dS_q(t,k,n;m)$, $m=n-p$, is
$$
\sum_{r= \max \{0,k-p\}}^{\min\{k,m\}} \sbinomq{m}{r} ~,
$$
i.e. $\sbinomq{m}{r}$ variables for all the $r$-subspaces of $\F_q^m$,
where $\max \{0,k-p\} \leq r \leq \min \{k,m\}$.
\end{lemma}
\begin{proof}
The range of $r$ is a direct consequence from Lemma~\ref{lem:r_dim} by noting that either $s$ gets the value of $t$ for a given $t$
and if $m < t$ then also $m < k$ and the value of $r$ is at most $m$. For each
$r$-subspaces of $\F_q^m$ we have one variable and hence there are $\sbinomq{m}{r}$
variables for each~$r$.
\end{proof}

\begin{cor}
If $m \leq t$, then the number of variables is equal to the number of equations,
for the $p$-punctured $q$-Steiner system $\dS_q(t,k,n;m)$. This number is equal to
$$
\sum_{e=0}^m \sbinomq{m}{e} ~.
$$
If the equations are linearly independent, then there is a unique solution to
the set of equations in this case (when the variables
ore not constrained). If the solution consists of nonnegative integers then the
$p$-punctured $q$-Steiner system $\dS_q(t,t+1,n;m)$ exists.
\end{cor}

\begin{cor}
If $m = t+2$, then the number of variables
for the $p$-punctured $q$-Steiner system $\dS_q(t,t+1,n;m)$, $m=n-p$, is equal to
$$
\sum_{r=0}^{t+1} \sbinomq{t+2}{r} ~.
$$
The number of equations in this case is equal to
$$
\sum_{s=0}^t \sbinomq{t+2}{s} ~.
$$
If we set the value of the variable which corresponds to the null 0-subspace of $\F_q^m$ to be
$\frac{\sbinomq{n-m}{t}}{\sbinomq{k}{t}}$ and the equations are linearly independent,
then there is a unique solution to the set of equations in this case (when the variables
ore not constrained). If the solution consists of nonnegative integers then the
$p$-punctured $q$-Steiner system $\dS_q(t,t+1,n;m)$ exists.
\end{cor}
\begin{proof}
If the variable related to the null 0-subspace of $\F_q^m$ is set to $\frac{\sbinomq{n-m}{t}}{\sbinomq{k}{t}}$, then
all the $\sbinomq{t+2}{1}$ variables related to the 1-subspaces of $\F_q^m$ are equal to 0. Therefore, the
number of equations in the new set of equations is
$$
\sum_{s=1}^t \sbinomq{t+2}{s} ~.
$$
The number of variables which are not assigned with values is this new set of equations is
$$
\sum_{r=2}^{t+1} \sbinomq{t+2}{r} ~.
$$
Clearly, these two summations are equal and the claim follows.
\end{proof}

\begin{cor}
If $m = t+1$, then the number of variables
for the $p$-punctured $q$-Steiner system $\dS_q(t,t+1,n;m)$, $m=n-p$, is equal to
$$
\sum_{r=0}^{t+1} \sbinomq{t+1}{r} ~.
$$
The number of equations in this case is equal to
$$
\sum_{s=0}^t \sbinomq{t+1}{s} ~.
$$
If we set the value of the variable which corresponds to the 0-subspace of $\F_q^m$ to be
$\frac{\sbinomq{n-m}{t}}{\sbinomq{k}{t}}$ and the equations are linearly independent, then there is a unique solution to
the set of equations in this case (when the variables
ore not constrained). If the solution consists of nonnegative integers then the
$p$-punctured $q$-Steiner system $\dS_q(t,t+1,n;m)$ exists.
\end{cor}

Note, that if there is a unique solution to the set of equations, then the existence of the
related design, i.e. $p$-punctured $q$-Steiner system $\dS_q(t,k,n;m)$ is not guaranteed yet.
Only if the unique solution is a nonnegative integer solution, then the design exists.
We also did not consider the linear independence of the equations, although it can be proved
in some cases. It is also important to understand that the number of equations and the number
of variables can be large, and in most cases the number of variables is much larger than
the number of equations. In this case there are many free variables, which usually make it even harder
to find if the set of equations have a solution with nonnegative integer values for the variables.

In the sequel we will examine cases, where the set of equations have a solution
with nonnegative integers. In these cases it will be proved that the
$p$-punctured $q$-Steiner system $\dS_q(t,k,n;m)$ exists. In most cases we will consider
\emph{uniform solutions}, i.e. solutions in which for each $r$, the number of $r$-subspaces
in the systems is equal for any two $r$-subspaces of~$\F_q^m$, i.e. the related variables have the same value. The
related design will be called a \emph{uniform design}. For such systems we can reduce
the number of variables and the number of equations. The choice of uniform solution is usually a good choice
when the equations are linearly independent. In such a case the solution is uniform
in many cases.

Let $\dS$ be a uniform $p$-punctured $q$-Steiner system $\dS_q(t,k,n;m)$, $m=n-p$.
Let $Z$ be an $r$-subspace of $\F_q^m$ and let $X_{r,m}$ be the number of appearances
of $Y$ in $\dS$.
The conclusion of our discussion is the following set of equations for uniform designs.

%

\begin{theorem}[\bf Equations for a uniform $p$-punctured $q$-Steiner system $\dS_q(t,k,n;m)$]
Let $\dS$ be a uniform $p$-punctured $q$-Steiner system $\dS_q(t,k,n;m)$, $m=n-p$.
For each $s$, $\max \{0,t-p\} \leq s \leq \min \{t,m\}$, the following equation must be satisfied.

$$
N_{(s,m),(t,n)} = \sum_{r=\max \{ k-p,s  \} }^{\min \{ k-t+s,m \} }  D_{s,r,m} \cdot C_{(s,t),(r,k)} \cdot X_{r,m}~.
$$
\end{theorem}
\begin{proof}
The left side of the equation is the number of distinct $t$-subspaces in $\F_q^n$
which are formed by extending a given $s$-subspace $Y$ of $\F_q^m$. The right hand side
is summing over all the $r$-subspaces of $\F_q^m$ which contain $Y$ (the range
is obtained from Lemma~\ref{lem:r_dim}), where $D_{s,r,m}$ is the number of $r$-subspaces
which contain $X$, $C_{(s,t),(r,k)}$ is the number of appearances of $Y$ in such a given $r$-subspace,
and $X_{r,m}$ is the number of appearances of each $r$-subspace in~$\dS$.
\end{proof}

\vspace{0.2cm}

\section{Examples for Existed Systems}
\label{sec:examples}

In this section we will give examples of $p$-punctured $q$-Steiner system
$\dS_q (t,k,n;m)$ for various parameters. We start with the 3-punctured
$q$-Fano plane $\dS_q(2,3,7;4)$ and continue with $\dS_q(3,4,8;4)$, $\dS_2(3,4,8;5)$,
$\dS_q(4,5,11;6)$, and $\dS_q(5,6,12;6)$. We conclude with a more general example for
the $k$-punctured $q$-Steiner system $\dS_q (3,4,2k;k)$, $k \equiv 2$ or $4~(\text{mod}~6)$, $k \geq 4$.

\vspace{1.4cm}

\noindent
{\bf The 3-punctured $q$-Steiner system $\dS_q(2,3,7;4)$:}

\vspace{0.1cm}

There are $\sbinomq{4}{s}$ equations
for each $0 \leq s \leq 2$, for a total of $1+(q^3+q^2+q+1)+(q^2+1)(q^2+q+1)$ equations.
There are $\sbinomq{4}{r}$ variables
for each $0 \leq r \leq 3$, for a total of $1+(q^3+q^2+q+1)+(q^2+1)(q^2+q+1) +(q^3+q^2+q+1)$ variables.

For $s=0$, there is a unique equation for the $0$-subspace (the null space) given by
$$
\sbinomq{3}{2} = (q^2+q+1)a +  b_{i_1}+b_{i_2} + \cdots + b_{i_{q^3+q^2+q+1}}~,
$$
where $a$ is the unique variable related to the $0$-subspace, while $b_{i_j}$ is
a variable for an $1$-subspace.
To have a unique solution we must have linearly independent equations in which
the number of variables equals the number of equations.
Hence, we set $a=1$ which implies that $b_{i_j}=0$ for each~$j$.

For $s=1$, there are $q^3+q^2+q+1$ equations related to the $1$-subspaces, where each equation is of the form
$$
q^2 \sbinomq{3}{1} =  (q^2+q)b + c_{i_1}+c_{i_2} + \cdots + c_{i_{q^2+q+1}}~,
$$
where $b$ is a variable related to an $1$-subspace and hence $b=0$, while $c_{i_j}$ is
a variable for a $2$-subspace.

For $s=2$, there are $(q^2+1)(q^2+q+1)$ equations for $2$-subspaces, where each equation is of the form
$$
q^6 \sbinomq{3}{0} = q^2 c + d_{i_1}+d_{i_2} + \cdots + d_{i_{q+1}}~,
$$
where $c$ is a variable related to a $2$-subspace, while $d_{i_j}$ is
a variable related to a $3$-subspace.

This system of equations has a unique solution, which is also a solution for
a uniform design (uniform punctured system), $X_{0,4}=1$, $X_{1,4}=0$, $X_{2,4}=q^2$, and $X_{3,4}=q^4(q-1)$.

\vspace{0.2cm}

\noindent
{\bf The 4-punctured $q$-Steiner system $\dS_q(3,4,8;4)$:}

\vspace{0.1cm}

There are $\sbinomq{4}{s}$ equations
for each $0 \leq s \leq 3$, and
there are $\sbinomq{4}{r}$ variables
for each $0 \leq r \leq 4$.
To have a unique solution, which also forms a uniform design, we set $X_{0,4}=1$
which implies that $X_{1,4}=0$, and the system of equations has
the unique solution, $X_{2,4}=q^2 (q^2+1)$, $X_{3,4}=q^4(q^4-1)$, and $X_{4,4}=q^{12}-q^{11}+q^7$.

\vspace{0.2cm}

\noindent
{\bf The 3-punctured $q$-Steiner system $\dS_q(3,4,8;5)$:}

\vspace{0.1cm}

It is left for the reader to verify that the following set $\T$ is a 3-punctured $q$-Steiner system $\dS_q(3,4,8;5)$.
contains:
\begin{enumerate}
\item One 1-subspace which is punctured into the unique 0-subspace of $\F_q^4$.

\item The $q^2(q^2+1)(q^2+q+1)$ distinct 2-subspaces of $\F_q^5$, which are
punctured into a 2-subspace of $\F_q^4$, each one is contained exactly once in $\T$.

\item The $(q^2+q+1)(q^2+1)$ distinct 3-subspaces of $\F_q^5$, which are
punctured into a 2-subspace of $\F_q^4$, each one is contained $q^4$ times in $\T$.

\item The $q^3 (q^3+q^2+q+1)$ distinct 3-subspaces of $\F_q^5$, which are
punctured into a 3-subspace of $\F_q^4$, each one is contained $q(q^3-1)$ times in $\T$.

\item The $q^3+q^2+q+1$ distinct 4-subspaces of $\F_q^5$, which are
punctured into a 3-subspace of $\F_q^4$, each one is contained $q^7 (q-1)$ times in $\T$.

\item The $q^4$ distinct 4-subspaces of $\F_q^5$, which are
punctured into the unique 4-subspace of $\F_q^4$, each one is contained $q^8-q^7+q^3$ times in $\T$.
\end{enumerate}

\noindent
{\bf The 5-punctured $q$-Steiner system $\dS_q(4,5,11;6)$:}

\vspace{0.1cm}

There are $\sbinomq{6}{s}$ equations
for each $0 \leq s \leq 4$ and
there are $\sbinomq{6}{r}$ variables
for each $0 \leq r \leq 5$.
To have a uniform design we set $X_{0,6}=1$ which implies that $X_{1,6}=0$ and the system of equations
will have a unique solution $X_{2,6}= q^2 (q^2+1)$, $X_{3,6}=q^9+q^7-q^4$,
$X_{4,6}=q^{14}-q^9+q^7$, and $X_{5,6}=(q^{18}+q^{11})(q-1)$.

\vspace{0.2cm}

\noindent
{\bf The 6-punctured $q$-Steiner system $\dS_q(5,6,12;6)$:}

\vspace{0.1cm}

There are $\sbinomq{6}{s}$ equations
for each $0 \leq s \leq 5$ and
there are $\sbinomq{6}{r}$ variables
for each $0 \leq r \leq 6$.
A solution for a uniform design for the system of equations is $X_{0,6}=1$, $X_{1,6}=0$,
$X_{2,6}=q^2 (q^4+q^2+q)$, $X_{3,6}=q^4 (q^8+q^6+q^5-1)$, $X_{4,6}=q^7 (q^{11}+q^9+q^7-q^6+1)$,
$X_{5,6}=q^{11} (q^{13}-q^7+q^6-1)$, and $X_{6,6}=q^{16} (q^{14}-q^{13}+q^7-q^6+1)$.

\vspace{0.2cm}

\noindent
{\bf The $k$-punctured $q$-Steiner system $\dS_q (3,4,2k;k)$, $k \equiv 2$ or $4~(\text{mod}~6)$, $k \geq 4$:}

\vspace{0.1cm}

In this case, we will consider only a possible uniform design.
For this design we have that $X_{0,k}=  \frac{\sbinomq{k}{3}}{\sbinomq{4}{3}}$,
$X_{1,k}=0$, $X_{2,k}=q^{k-2} \frac{q^k-1}{q^2-1}$, $X_{3,k}=q^k (q^k-1)$, and
$X_{4,k}=\frac{(q^{3k}-q^{2k+3}+q^{k+3})(q-1)}{q^{k-3}-1}$. We note that the reminder in the division of the polynomials
in $X_{4,k}$ is $q^7 - q^6$ and hence $X_{4,k}$ is an integer only for
$k=4$ and all $q$'s. This solution was given in a previous example for $\dS_q(3,4,8;4)$.

\vspace{0.2cm}

\noindent
{\bf The $k$-punctured $q$-Steiner system $\dS_q (2,3,2k+1;k+1)$, $k \equiv 1$ or $3~(\text{mod}~6)$, $k \geq 3$:}

\vspace{0.1cm}

The number of equations in the system is $\sum_{s=0}^2 \sbinomq{k+1}{s}$. The number of variables
is $\sum_{r=0}^3 \sbinomq{k+1}{r}$. We will consider only uniform designs and hence we only have 3 equations
and 4 variables.

The first equation for the $0$-subspace of $\F_q^{k+1}$ is $N_{(0,k+1),(2,2k+1)}= D_{0,0,k+1} \cdot C_{(0,2)(0,3)} \cdot X_{0,k+1}
+ D_{0,1,k+1} \cdot C_{(0,2)(1,3)} \cdot X_{1,k+1}$
which is equal to $\sbinomq{k}{2} =  \sbinomq{3}{2} \cdot X_{0,k+1}+ \sbinomq{k+1}{1} \cdot X_{1,k+1}$. If we set
$X_{0,k+1} = \frac{\sbinomq{k}{2}}{\sbinomq{3}{2}}$, then we have $X_{1,k+1}=0$.

The second equation for 1-subspaces is $N_{(1,k+1),(2,2k+1)}= D_{1,1,k+1} \cdot C_{(1,2)(1,3)} \cdot X_{1,k+1}
+ D_{1,2,k+1} \cdot C_{(1,2)(2,3)} \cdot X_{2,k+1}$.
Since $N_{(1,k+1),(2,2k+1)}=q^{k-1} \sbinomq{k}{1} = q^{k-1} \frac{q^k-1}{q-1}$ and
$D_{1,2,k+1} = \sbinomq{k}{1}= \frac{q^k-1}{q-1}=q^{k-1}+q^{k-2}+\cdots +q+1$,
it follows that
$$
q^{k-1} \sbinomq{k}{1} = q^{k-1} \frac{q^k -1}{q-1} =  (q^2+q) X_{1,k+1}+ (q^{k-1}+q^{k-2}+\cdots +q+1 ) X_{2,k+1} ~.
$$

The third equation for 2-subspaces is
$N_{(2,k+1),(2,2k+1)}= D_{2,2,k+1} \cdot C_{(2,2)(2,3)} \cdot X_{2,k+1} + D_{2,3,k+1} \cdot C_{(2,2)(3,3)} \cdot X_{3,k+1}$.
Since $N_{(2,k+1),(2,2k+1)}=q^{2k}$ and $D_{2,3,k+1} = \sbinomq{k-1}{1}= \frac{q^{k-1}-1}{q-1}=q^{k-2}+q^{k-3}+\cdots +q+1$,
it follows that
$$
q^{2k} =  q^2 X_{2,k+1} + (q^{k-2}+q^{k-3}+\cdots +q+1) X_{3,k+1} ~.
$$

The solution for this set of equations is $X_{0,k+1}= \frac{\sbinomq{k}{2}}{\sbinomq{3}{2}}$,
$X_{1,k+1}=0$, $X_{2,k+1}=q^{k-1}$, and $X_{3,k+1}=q^{k+1} (q-1)$.

\section{A Recursive Construction}
\label{sec:recursion}

In this section we present a recursive construction for a $p$-punctured
$q$-Steiner system $\dS_q (2,3,2k+1;k+1 +\lfloor \frac{k+1}{3} \rfloor )$,
$p=k- \lfloor \frac{k+1}{3} \rfloor$, where
$k \equiv 1$ or $3~(\text{mod}~6)$. The basis for the construction
is the trivial $q$-Steiner system $\dS_q(2,3,3)$.

Let $k \equiv 1$ or $3~(\text{mod}~6)$, which implies that $2k+1 \equiv 3$ or $7~(\text{mod}~12)$,
and assume that there exists a $p$-punctured $q$-Steiner system $\dS_q(2,3,k;\lfloor \frac{k+1}{3} \rfloor)$, $p=k- \lfloor \frac{k+1}{3} \rfloor$.
Let $\dS$ be a $k$-punctured $q$-Steiner system $\dS_q(2,3,2k+1;k+1)$ presented in Section~\ref{sec:examples}.
For $\dS$ we have that $X_{0,k+1}= \frac{\sbinomq{k}{2}}{\sbinomq{3}{2}}$,
$X_{1,k+1}=0$, $X_{2,k+1}=q^{k-1}$, and $X_{3,k+1}=q^{k+1} (q-1)$.
In the recursive construction, we will generate a system $\T$, a $p$-punctured
$q$-Steiner system $\dS_q(2,3,2k+1;k+1 + \lfloor \frac{k+1}{3} \rfloor )$, $p=k- \lfloor \frac{k+1}{3} \rfloor$.

Let $r= \lfloor \frac{k+1}{3} \rfloor$ be the number of columns that
should be appended to the subspaces
(of dimension 0, 2, and 3) of $\dS$ to form $\T$.
To each one of the $\sbinomq{k+1}{3}$ distinct $3$-subspaces of
$\dS$ we append the $q^{3r}$ possible combinations of $r$ columns.
Each column has $q^3$ possible combinations by Lemma~\ref{lem:addNone}.
Since $X_{3,k+1}=q^{k+1} (q-1)$, it follows that each such combination
(a $3$-subspace of $\F_q^{k+1+r}$), whose $r$-punctured subspace is also a 3-subspace,
will appear $q^{k+1-3r} (q-1)$ times in $\T$.
To the $\frac{\sbinomq{k}{2}}{\sbinomq{3}{2}}$ $0$-subspaces of $\dS$ we append
the subspaces of a $(k-r)$-punctured $q$-Steiner system $\dS_q(2,3,k;r)$ system which
exists by our assumption. Hence, we have completed
the extension of the 0-subspaces and 3-subspaces of $\dS$. To complete our construction we have
to extend the 2-subspaces of $\dS$.

For the extension of the 2-subspaces we need two more concepts, namely spreads and large sets in $\cG_q(k+1,2)$
(known as 1-spreads and 1-parallelisms in PG($k,q$) ~).
A \emph{spread} in $\cG_q(k+1,2)$ is a set of 2-subspaces
whose nonzero elements form a partition of all the elements of $\F_q^{k+1} \setminus \{ 0 \}$, i.e.
each nonzero vector of $\F_q^{k+1}$ appears in exactly one 2-subspace of the spread.
In other words, a spread in $\cG_q(k+1,2)$ is a $q$-Steiner system $\dS_q(1,2,k+1)$.
A \emph{large set} (1-parallelism) of $q$-Steiner systems $\dS_q(1,2,k+1)$ is a partition
of all 2-subspaces of $\cG_q(k+1,2)$ into $q$-Steiner systems $\dS_q(1,2,k+1)$ (spreads).
If $q=2$, then such large sets are known to exist whenever $k+1$ is even~\cite{Bak76}

We continue by considering the case of $q=2$.
Note, that $k+1$ is even and hence there exists a spread in $\cG_2(k+1,2)$.
The size of such spread is $\frac{2^{k+1}-1}{3}$, i.e. it contain
$\frac{2^{k+1}-1}{3}$ subspaces. The total number of
subspaces in $\cG_2 (k+1,2)$ is $\sbinomtwo{k+1}{2} = \frac{(2^{k+1}-1)(2^k-1)}{3}$.
There exists a partition (large set) of these 2-subspaces into disjoint spreads and
hence there are $2^k-1$ disjoint spreads in such a large set. We continue and
arbitrarily partition these $2^k-1$ disjoint spreads into $2^r$ sets of spreads, one set with
$2^{k-r}-1$ spreads and $2^r-1$ sets each one with $2^{k-r}$ spreads.
To each one of these $2^r-1$ sets we assign arbitrarily a different nonzero row vector
of length $r$, and the all-zero vector of length $r$ is assigned to the set of size $2^{k-r}-1$.

For demonstration of the construction, each 2-subspace of
$\dS$ is represented by
a $3 \times (k+1)$ matrix, each 2-subspace of
$\T$ is represented by
a $3 \times (k+1+r)$ matrix, and  each 3-subspace of
$\T$ is represented by
a $7 \times (k+1+r)$ matrix.

Consider now these two sets of spreads:
\begin{enumerate}
\item For the set which contains $2^{k-r}-1$ spreads, each 2-subspace $X$ from
each spread is contained $2^{k-1}$ times in the
$\dS$. The 2-subspaces
$X$ is extended to several 2-subspaces in $\F_2^{k+1+r}$ as follows.
The first $k+1$ columns which represent these 2-subspaces are equal to the
$3 \times (k+1)$ matrix which represents $X$. In the last $r$ columns
there are 4 possible options in each column and thus $2^{2r}$ distinct combinations
of $r$ columns. Each such combination will appear $2^{k-1-2r}$ times in $\T$.

\item For a set which contains $2^{k-r}$ spreads (there are $2^r-1$ such sets in the
partition), each 2-subspace $X$ from each
spread is contained $2^{k-1}$ times in $\dS$.
There is a nonzero vector $v$ of length $r$ which is assigned to this set.
The 2-subspace $X$ is extended to several 3-subspaces in $\F_2^{k+1+r}$ as follows.
The first three rows in the first $k+1$ columns which represent these subspaces are equal to the
$3 \times (k+1)$ matrix which represents $X$.
The next three rows in these $k+1$ columns are also equal to the
$3 \times (k+1+r)$ matrix which represents $X$. The seventh and the last
row in these $k+1$ columns is a row of \emph{zeroes}.
We turn now to complete the last $r$ columns in the $7 \times (k+1+r)$ matrices
which represents the 3-subspaces extended from $X$. The entries of the last (seventh) row in these columns
are assigned with the values of $v$. The first column in which $v$ has a \emph{one}
has values which corresponds to the unique extension from a 2-subspace to a 3-subspace
as proved in Lemma~\ref{lem:addOne}. Finally, in each other column there are 4 possible distinct
combinations: if the related entry in $v$ is a \emph{zero}, it relates to the extension
from 2-subspace to 2-subspace; and if the related entry in $v$ is a \emph{one}
it relates to the extension from 3-subspace to 3-subspace in which there are
4 combinations, out of the 8 combinations, with a \emph{one} in a given coordinate.
In total there are $2^{2(r-1)}$ distinct combinations for these $r$ columns.
Each such combination will appear $2^{k-1-2(r-1)}$ times in $\T$.
\end{enumerate}

For a proof that $\T$ is $p$-punctured
$q$-Steiner system $\dS_q(2,3,2k+1;k+1 + r )$, $p=k-r$,
follows immediately from the described construction.
The major steps of the proof will be given in a the specific case
of $\dS(2,3,7;5)$ in Section~\ref{sec:2punctured}.

Generalization for $q>2$ is similar, but the requirement is the existence of large set
of $q$-Steiner system $\dS_q(1,2,k+1)$, where $k \equiv 1$ or $3~(\text{mod}~6)$.
Such large set is known to exist for $q>2$ only if
$k+1$ is a power of 2~\cite{Beu74}, making the possible generalizations for $q>2$
with limited number of parameters. An example of this construction for general $q$
and $2k+1=7$ is given in Section~\ref{sec:2punctured}.
The recursive construction, with the basis of $\dS_q(2,3,3)$, leads to the following theorem.

\begin{theorem}
There exists a $p$-punctured
$q$-Steiner system $\dS_q (2,3,2^\ell -1;2^\ell -1 -\lfloor \frac{2^\ell -1}{3} \rfloor )$,
$p=\lfloor \frac{2^\ell -1}{3} \rfloor$, $\ell \geq 3$.
\end{theorem}

For $q=2$ the construction can be applied also starting with the $q$-Steiner system $\dS_2 (2,3,13)$~\cite{BEOVW}.

\section{The structure of the $q$-Fano plane}
\label{sec:structure}

In this section, we present the structure of the $q$-Fano plane (if exists)
based on its punctured designs.
The $q$-Fano plane $\dS_2(2,3,7)$ is the one on which most research was done in the past, e.g.~\cite{BKN15,EtVa11a,HeSi16,Tho96}.
The size of the $q$-Fano plane for $q=2$
is smaller and hence with the mentioned figures for various substructures of
the $q$-Fano plane, one can take it as a toy example to
try and construct it by hand, needless to say it might be easier to check its existence with computer search.
Finally, note that sometimes we have to consider for $q>2$ 1-subspaces, instead of vectors for $q=2$.

Throughout our discussion, let $\dS$ be a $q$-Steiner system $\dS_q(2,3,7)$.
We start with a uniform solution for the 3-punctured $q$-Steiner system
$\dS_q(2,3,7;4)$. Such a uniform solution, given in Section~\ref{sec:examples} implies that $X_{0,4}=1$ which implies
that $X_{1,4}=0$, $X_{2,4}=q^2$, and $X_{3,4}=q^4(q-1)$. W.l.o.g. we can set $X_{0,4}=1$,
since in any system, w.l.o.g. one subspace can be chosen. Furthermore, in this case
$X_{0,4}=1$ implies that the design is uniform. $X_{0,4}=1$ implies that the 3-subspace whose
first four columns are all-zero columns is contained in $\dS$. Let $Z_1$ denote this
3-subspace in which the first four columns are all-zero. For symmetry let $Z_2$ denote
the 3-subspace of $\F_q^7$ in which the last four columns are all-zero columns.

In $\dS$, each 1-subspace of $\F_q^7$ is contained in exactly $\frac{q^6-1}{q-1}$ 3-subspaces.
By puncturing the last coordinate of each 3-subspace of $\dS$,
all the 3-subspaces which contain the vector 0000001 will be punctured into
a spread, i.e. a $q$-Steiner system $\dS_q(1,2,6)$.

Next, in our exposition we exclude $Z_1$ from $\dS$ for the current paragraph.
Each 3-subspace of $\dS$ which contains a nonzero vector which starts with four \emph{zeroes}
is 3-punctured into a 2-subspace of $\F_q^4$. There are $\frac{q^3-1}{q-1}$ 1-subspaces which contain such vectors, each one is contained
in $q^2 (q^2+1)$ distinct 3-subspaces of $\dS \setminus \{ Z_1 \}$ (since the only 3-subspace of $\dS$ which
contains two such 1-subspaces is $Z_1$; thus, clearly two of them cannot be contained together in the same 3-subspace of $\dS \setminus \{ Z_1 \}$)
for a total of $q^2 (q^2+1) (q^2+q+1)$ such 3-subspaces which are punctured
into the $q^2 (q^2+1) (q^2+q+1)$ (non-distinct) 2-subspaces
of the 3-punctured $q$-Steiner system $\dS_q(2,3,7;4)$ derived from $\dS$.
Since each 1-subspace of $\F_q^7$ is contained in exactly one 3-subspace with each
other 1-subspace of $\F_q^7$, it follows that each such 1-subspace is responsible
for exactly $q^2 (q^2+1)$ 2-subspaces (some of them are identical) of the 3-punctured $q$-Steiner system $\dS_q(2,3,7;4)$.
Hence, each nonzero prefix of length 4 of vectors of $\F_q^7$ (there are $q^4-1$ such nonzero prefixes) appears $\frac{q^2 (q^2+1) \cdot (q^2-1)}{q^4-1}=q^2$
times in these $q^2 (q^2+1)$ 2-subspaces. There are many possible partitions to obtain such $q^2+q+1$ 3-punctured sets
of $q^2 (q^2+1)$ 2-subspaces for this purpose
(each part in this partition is obtained by puncturing
three times $q^2 (q^2+1)$ 3-subspaces of $\F_q^7$ which contain the same vector which starts with four \emph{zeroes}).
One suggestion was given in the recursive construction of Section~\ref{sec:recursion}.
It will be discussed again in Section~\ref{sec:2punctured}.

We continue by imposing w.l.o.g. a certain structure on $\dS$.
Many structures can be imposed, each one will imply a different specification for $\dS$. The one that we suggest
now seems to be one of the most successful ones for possible construction of the
$q$-Fano plane. We already forced the 3-subspace $Z_1$
to be a 3-subspace in $\dS$. Now, we assume
that w.l.o.g. also the 3-subspace whose last four columns are all-zero, i.e. $Z_2$, is also a 3-subspace in $\dS$.

Next, we will show why w.l.o.g. it can be assumed that $Z_2 \in \dS$.
In the 3-punctured $q$-Steiner system $\dS_q(2,3,7;4)$ each 3-subspace of $\F_q^4$ appears exactly $q^4 (q-1)$ times.
There are $\frac{q^4-1}{q-1}$ such 3-subspaces of $\F_q^4$, for $q^3$ of them, the first three columns
form the unique 3-subspace in $\F_q^3$, and these three columns can be followed by the all-zero column (which is one
of the $q^3$ possible extensions of the unique 3-subspace of $\F_q^3$ by Lemma~\ref{lem:addNone}). Let $X$
be a 3-subspace of $\dS$ which is 3-punctured to these four columns. To continue
our discussion, we need the following simple lemma

\begin{lemma}
\label{lem:linear_comb}
If $\hat{\dS}$ is a $q$-Steiner system $\dS_q(t,k,n)$, then the system obtained by replacing
the $j$th column (for any $j$, $1\leq j \leq n$) in all the $k$-subspaces of $\hat{\dS}$,
by a linear combination of columns, which contain the $j$th column (in any nonzero
multiplicity), is also a $q$-Steiner system $\dS_q(t,k,n)$.
\end{lemma}

Therefore, since
the first three columns of $X$ have rank three it follows that we can
form some specific three linear combinations, containing the 5th, the 6th, and the 7th column of $X$, respectively.
Each such linear combination will sum
to \emph{zero} for the related column of $X$.
We replace the 5th, 6th, and 7th columns of $X$
with these linear combinations, i.e. these columns are now all-zero columns in a 3-subspace which replaces $X$.
These three linear combinations are performed and replace the related columns
in all the $(q^6+q^5+q^4+q^3+q^2+q+1)(q^2-q+1)$ 3-subspaces of $\dS$.
By abuse of notation we call the new system also $\dS$. We note that after this was done, $Z_1$ was not
affected and it remains a 3-subspace of $\dS$. Hence, the two 3-subspaces
(starting with four all-zero columns and ending with such four columns, i.e. $Z_1$ and $Z_2$) can be
forced to be in $\dS$ which we do. As a consequence, all the consequences that we will derive
regarding the first $\ell$, $1 \leq \ell \leq 6$, columns of the 3-subspaces in $\dS$,
are also correct consequences concerning the last $\ell$ columns of these 3-subspaces.
Let $\T$ be the system formed from $\dS$ by performing puncturing three times
on the first three columns of all the 3-subspaces of $\dS$ (note that $\T$ is isomorphic to $\dS_q(2,3,7;4)$, but such system
was defined before only when the last columns are punctured).

Each pair of 1-subspaces of $\F_q^7$ which contain vectors which start with four \emph{zeroes} and vectors which end with four
\emph{zeroes} appear together in exactly one 3-subspace of $\dS$. There are no three such linearly independent vectors
in the same 3-subspace of $\dS$ since two such vectors (with either four leading \emph{zeroes} or four \emph{zeroes} at the tail)
will sum to another such vector and the result will be a 1-subspace in either $\dS_q(2,3,7;4)$ or $\T$, a contradiction.
Therefore, there are exactly $\frac{q^3-1}{q-1} \cdot \frac{q^3-1}{q-1}=(q^2+q+1)^2$ 3-subspaces which contain 1-subspace which has
a vector with four leading \emph{zeroes} and one 1-subspace which contains a vector which has four \emph{zeroes} at the tail.
Let $\A$ be the set of 3-subspaces of $\dS$ which form the $q^2 (q^2+1)(q^2+q+1)$ 2-subspaces in $\dS_q(2,3,7;4)$ and
let $\dB$ be the set of 3-subspaces of $\dS$ which form the $q^2 (q^2+1)(q^2+q+1)$ 2-subspaces in $\T$.
Clearly,
$$
|\A|=|\dB|=q^2 (q^2+1)(q^2+q+1),~~~ |\A \cap \dB|=(q^2+q+1)^2,~~~|\A \setminus\dB|=|\dB \setminus \A|=(q^2+q+1)(q^4-q-1)~~.
$$

Therefore, there are $(q^6+q^5+q^4+q^3+q^2+q+1)(q^2-q+1)-(2 \cdot (q^2+q+1)(q^4-q-1) + (q^2+q+1)^2+1+1)=q(q^7-q^5-q^4-2q^3+q^2+2q+2)$
3-subspaces in $\dS$ in which the projection
on the first four columns yields a 3-subspace of $\F_q^4$ and the projection
on the last four columns yields a 3-subspace in $\F_q^4$.

Finally, as we mentioned before, there are other possible 3-subspaces that can be imposed on $\dS$
in addition to $Z_1$. We will briefly mention one more such option. We are mainly interested in
3-subspaces which have four all-zero columns since we know the structure of the related
design formed by puncturing the three other columns. We claim that we can impose on $\dS$ to have
three such 3-subspaces (two of them are $Z_1$ and $Z_2$).
Let $Z_3$ be such 3-subspace that has all-zero columns in columns 1, 2, 6, and 7.
The proof that we can force $Z_1$, $Z_2$, and $Z_3$ to be together in $\dS$
is very similar to the one which forced $Z_1$ and $Z_2$ to be together in $\dS$.
For this purpose we consider the $(q^2+q+1)^2$ subspaces of $\A \cap \dB$. By considering
$\dS_q(2,3,7;1)$, we have that there are exactly
$q^5+q^3+q^2+1$ 3-subspaces in $\dS$ whose 4-th column is the all-zero column. We already proved that
if the 3-punctured $q$-Steiner system $\dS_q(2,3,7;4)$ contains the 0-subspace, then the system
is uniform and each 3-subspace is contained $q^4 (q-1)$ times. In one such 3-subspace the 4-th
column is the all-zero column. Hence, the other $q^5+q^3+q^2+1 -1 - q^4(q-1) =q^4+q^3+q^2$ 3-subspaces
with all-zero 4-th column are in $\A$. Since, $|\A \cap \dB|=(q^2+q+1)^2 > q^4+q^3+q^2$,
it follows that there exists a subspace in $\A \cap \dB$ whose 4-th column is a nonzero vector.
Now, we can permute the columns in all the system as follows. Columns 1, 2, and 3 are permuted in a way
that columns 3 and 4 in $Y$ will be linearly independent. Columns 5, 6, and 7 are permuted in a way
that columns 4 and 5 in $Y$ will be linearly independent.
Now, Lemma~\ref{lem:linear_comb} is applied to have all-zero columns 1 and 2 in $Y$ by using linear combinations
with columns 3 and 4. Similarly, Lemma~\ref{lem:linear_comb} is applied
to have all-zero columns 6 and 7 in $Y$ by using linear combinations
with columns 4 and 5. Note, that these operations do not affect $Z_1$ and $Z_2$.
The consequence is that the $q$-Steiner system $\dS_q(2,3,7)$ contains
$Z_1$, $Z_2$, and $Z_3$.

Can we have another 3-subspace in $\dS$ with four all-zero columns? We cannot give a definite answer to this
question. In such a 3-subspace, two all-zero columns must be in the first three columns, say columns 1 and 3, and the
two other all-zero columns in the last three columns, say columns 5 and 7.
Similarly, a fifth 3-subspace with four all-zero columns might be added.

Based on the forced structure described in this section, one can start a computer search to construct
the $q$-Fano plane for $q=2$. The outcome of such search is of great interest. We believe that the
structure that we found will make it easier to perform such a search.

\section{The 2-punctured $q$-Steiner system $S_q(2,3,7;5)$}
\label{sec:2punctured}

In this section we continue and present a possible structure for the $q$-Fano plane, namely,
we present a construction of a 2-punctured $q$-Steiner system $\dS_q(2,3,7;5)$.
We note that this is a possible substructure of the $q$-Steiner system $\dS_q(2,3,7)$
(first five columns of the system), but it is not forced like the systems described in Section~\ref{sec:structure},
and hence, it might not be possible to complete the constructed design into the $q$-Fano plane, even
if the related $q$-Fano plane exists. The construction is based on extensions for all the
subspaces of the 3-punctured $q$-Steiner system $\dS_q(2,3,7;4)$.

Let $\dS$ be a uniform 3-punctured $q$-Steiner system $\dS_q(2,3,7;4)$ with the uniform solution,
found in Section~\ref{sec:examples}, i.e. $X_{0,4}=1$, $X_{1,4}=0$, $X_{2,4}=q^2$ and $X_{3,4}=q^4(q-1)$.
We use four types of extensions as follows:

\noindent
{\bf Type 1:}
The unique 0-subspace of $\dS$ is extended in a unique way to a 1-subspace of $\F_q^5$

\noindent
{\bf Type 2:}
Each 3-subspace of $\F_q^4$ in $\dS$ can be extended in $q^3$ different ways (see Lemma~\ref{lem:addNone}).
We use each such extension $q(q-1)$ times in $S_q(2,3,7;5)$, i.e. each one of the
$\sbinomq{4}{3} q^3$ such 3-subspaces of $\F_q^5$ will appear $q(q-1)$ times in our
constructed $\dS_q(2,3,7;5)$.

\vspace{0.5cm}

There are $q^2 (q^2+1)(q^2+q+1)$ 2-subspaces in $S_q(2,3,7;4)$, i.e. $(q^2+1)(q^2+q+1)$
distinct 2-subspaces for which each one appears $q^2$ times in $\dS_q(2,3,7;4)$.
From this set of 2-subspaces there are $q^2 (q^2+1)q^2$ 2-subspaces which will be extended
to 3-subspaces of $\F_q^5$ and $q^2 (q^2+1) (q+1)$ which will be extended to 2-subspaces of $\F_q^5$.
There are a total of $(q^2+1)(q^2+q+1)$ distinct 2-subspaces in $\F_q^4$ that can be partitioned
into $q^2+q+1$ disjoint spreads, each one of size $q^2+1$~\cite{Bak76,Beu74}.
We partition these disjoint spreads into two sets, one set $A$ will
contain $q^2$ spreads and a second set $B$ will contain $q+1$ spreads.

\noindent
{\bf Type 3:}
Each 2-subspace in $\dS$ which is contained in a spread from the
set $A$ is extended in a unique way (see Lemma~\ref{lem:addOne}) to a 3-subspace in $S_q(2,3,7;5)$.
Thus, each such 3-subspace of $\F_q^5$ (there are $(q^2+1) q^2$ such
3-subspaces) will appear $q^2$ times in our constructed $S_q(2,3,7;5)$.

\noindent
{\bf Type 4:}
Each 2-subspace in $\dS$ which is contained in a spread from the
set $B$ is extended in a $q^2$ ways (see Lemma~\ref{lem:addNone}) to 2-subspaces in $S_q(2,3,7;5)$.
Thus, each one of these $q^2$ new 2-subspaces of $\F_q^5$ (there are $(q+1)(q^2+1)q^2$
such 2-subspaces) will appear exactly once in our constructed $S_q(2,3,7;5)$.

The proof that the constructed system is indeed a 2-punctured $q$-Steiner system $S_q(2,3,7;5)$
will be sketched now. First note that the 2-subspaces of $\F_q^7$ are 2-punctured into
the unique 0-subspace, the $\sbinomq{5}{1}$ one-subspaces, and the $\sbinomq{5}{2}$ two-subspaces
of $\F_q^5$. There is a unique way to extend the 0-subspace into 2-subspace of $\F_q^7$.
By Lemmas~\ref{lem:addNone} and~\ref{lem:addOne} there are $q^4$ different ways to
extend each two-subspace of $\F_q^5$ into a two-subspace of $\F_q^7$ and $q^2+q$
different ways to extend each one-subspace of $\F_q^5$ into a two-subspace of $\F_q^7$.
Hence, to complete the proof we have to show that each such subspace (0-subspace, one-subspace, or two-subspace)
of~$\F_q^5$ appears in the constructed system this required amount of times.
We will distinguish between four cases.

\noindent
{\bf Case 1:} The unique 0-subspace of $\F_q^5$ has a unique extension to a two-subspace
of $\F_q^7$ and it is covered by the subspace of Type 1.

\noindent
{\bf Case 2:} Type 1 also provides the $q^2 + q$ copies of the 1-subspace of $\F_q^5$
whose first four columns are zeroes. The other 1-subspaces of $\F_q^5$ can be obtained
only from Type 3 or Type~4. Each two-subspace of $\F_q^4$,
contains $\sbinomq{2}{1} =q+1$ one-subspaces of $\F_q^4$. Each spread contains each
such one-subspace exactly once. Each such one-subspace is extended to an one-subspace
of $\F_q^5$ only if the related two-subspace of $\F_q^4$ is extended to a two-subspace of $\F_q^5$.
This is done only in Type 4 (from the spreads of $B$). $B$ contains $q+1$ different spreads, each one
has $q^2$ identical copies. Each one-subspace of $\F_q^4$ appears in all these spreads $q^2 (q+1)$ times.
Since each such one-subspace of $\F_q^4$ is extended in $q$ different ways to one-subspaces of $\F_q^5$,
we have that each one-subspace of $\F_q^5$, whose first four columns are not all all-zero columns,
appears $q^2+q$ times in our system as required.

\noindent
{\bf Case 3:} We examine first the two-subspaces of $\F_q^5$ whose first four columns form
one-subspaces. Each one should appear $q^4$ times in our system. These two-subspaces are
formed only in Type 3, where two-subspaces from the spreads of $A$ are extended into 3-subspaces
of $\F_q^5$ (and the contained one-subspace are extended to
two-subspaces). There are $q^2$ spreads in~$A$, each one-subspace of $\F_q^4$ appears exactly
once in each one of them. There are $q^2$ identical copies for each such spread, so each
one-subspace of $\F_q^4$ appear $q^4$ times in these spreads. They are extended in a unique
way to two-subspaces of $\F_q^5$ and hence each appears $q^4$ times in our system as required.

\noindent
{\bf Case 4:}
For the two-subspaces of $\F_q^5$ which are extended from two-subspaces of $\F_q^4$
we can do similar counting. This is unnecessary as it is easy to see that they are equally
distributed in Type 2 and Type 3. They appear the required $q^4$ times since the other
subspaces were proved to appear the required number of times and these just complete the
total numbers which is dictated from the $q$-Steiner system $\dS_q(2,3,7)$.

\section{Conclusion}
\label{sec:conclude}

We have presented a new framework to examine the existence of $q$-Steiner systems.
Based on this framework we have defined a new set of $q$-designs which are
punctured $q$-Steiner systems. Necessary conditions for the existence of such designs were presented.
Several parameters where these new designs exist, were given. A recursive
construction for one set of parameter for such designs was given. For future research we would
like to find more properties of this framework. The main problem in this direction
is to find lower bounds on $m$ for any given $(n-m)$-punctured
$q$-Steiner system $\dS_q(t,k,n;m)$.

We have used the new framework as a basis to determine whether a $q$-Fano
plane, i.e. a $q$-Steiner system $\dS_q(2,3,7)$ exists. For small values of $q$,
probably only for $q=2$ this might help to determine the existence of such system
by using computer search. A short step, rather than a complete solution to the problem,
to make a progress in solving the existence problem, can be done in one of
the following directions:

\begin{enumerate}
\item Find a punctured $q$-Steiner system $\dS_q(2,3,7;6)$. First step in this
direction would be to consider $q=2$.

\item The subspaces of the set $\A \cup \dB$ might be a key for the whole construction
of $\dS_q (2,3,7)$. A possible first step might be to find the 231 subspaces of this set for $q=2$
and to extend this set of 231 3-subspaces with as many as possible more 3-subspaces (say $M$ 3-subspaces),
such that no 2-subspace of $\F_2^7$ appears in more than one of the $231+M$ 3-subspaces.
Of course, the two subspaces with four all-zero columns in the first
or last columns must be included in the $231+M$ 3-subspaces.

\item Another small step forward will be to settle the possibility of more than three
3-subspaces with four all-zero columns. It is either to prove that no more than three (four or five)3-subspaces
with four all-zero columns cannot exist in $\dS_q(2,3,7)$ or to prove that w.l.o.g. we can assume the existence
of four or five such 3-subspaces.
\end{enumerate}
\vspace{0.5cm}

Last, but certainly not the least, we can report that there
is a breakthrough that was made towards a construction
of a $q$-Fano plane in January 2017. Niv Hooker~\cite{Hoo17} has found
a punctured $q$-Steiner system $\dS_q(2,3,7;6)$ for $q=2$, by using the method developed in this
paper. This system is very interesting. It consists of a 21 2-subspaces of $\F_2^6$, which form
a spread. Each 2-subspace of $\F_2^6$ which is not part of the spread is contained in exactly four
3-subspaces of the system. This new finding give us a renew hope to construct a $q$-Fano plane.

\begin{center}
{\bf Acknowledgments}
\end{center}
The author want to thank anonymous reviewers whose comments taught him where
the earlier version of this paper needs to be amended. The final outcome wouldn't
have been possible without their important comments.
The author would also like to thank the COST Action IC1104 on "Random Network Coding
and Designs Over GF(q)" for providing funding for workshops on this topic of research.


\end{document}